\RequirePackage{fix-cm}
\documentclass[11pt]{amsart}
\usepackage[utf8]{inputenc}
\usepackage{graphicx}
\usepackage{amsmath,amssymb,a4wide} 
\usepackage{color}
\usepackage{mathrsfs}
\usepackage{dsfont}
\usepackage{ulem}
\usepackage{mathabx}
\usepackage{tikz,pgfplots}
\newtheorem{theorem}{Theorem} 
\newtheorem{lemma}[theorem]{Lemma}
\newtheorem{prop}[theorem]{Proposition} 
\newtheorem{remark}[theorem]{Remark} 
\newtheorem{corollary}[theorem]{Corollary}
\newtheorem{defi}[theorem]{Definition}

\numberwithin{equation}{section}

\newcommand{\R}{{\mathbb R}} 
\newcommand{\C}{{\mathbb C}} 
\newcommand{\N}{{\mathbb N}}

\newcommand{\dd}{{\rm d}}

\usepackage{color}

\begin{document}

\title{High order linearly implicit methods for evolution equations}

\author{Guillaume Dujardin \and Ingrid Lacroix-Violet}
\address[I. Lacroix-Violet]{Universit\'e de Lorraine, CNRS, IECL, F-54000 Nancy, France}
\address[G. Dujardin]{Univ. Lille, Inria, CNRS, UMR 8524 - Laboratoire Paul Painlev\'e F-59000 Lille}

\date{Received: date / Accepted: date}

\maketitle


\begin{abstract}
  This paper introduces a new class of numerical methods for the time integration of evolution
  equations set as Cauchy problems of ODEs or PDEs.
  The systematic design of these methods mixes the Runge--Kutta collocation formalism
  with collocation techniques,
  in such a way that the methods are linearly implicit and have high order.
  The fact that these methods are implicit allows to avoid CFL conditions when the
  large systems to integrate come from the space discretization of evolution PDEs. Moreover,
  these methods are expected to be efficient since they only require to solve one linear system of equations
  at each time step,
  and efficient techniques from the literature can be used to do so.
  After the introduction of the methods, we set suitable definitions of consistency and stability
  for these methods. This allows for a proof that arbitrarily high order linearly implicit
  methods exist and converge when applied to ODEs.
  Eventually, we perform numerical experiments on ODEs and PDEs that illustrate our theoretical
  results for ODEs, and compare our methods with standard methods for several evolution PDEs.
\end{abstract}


\maketitle
{\small\noindent 
{\bf AMS Classification.} 65M12, 65M70, 65L20, 65L06, 81Q05, 35Q41, 35K05

\bigskip\noindent{\bf Keywords.} Cauchy problems, evolution equations, time integration,
numerical methods, high order, linearly implicit methods.}

\section{Introduction}
The goal of this paper is to introduce a new class of methods for the time integration
of evolution problems, set as (systems of) deterministic ODEs or PDEs.
This class consists in methods of arbitrarily high order that require only the solution
of one linear problem at each time step: no nonlinear system is to be solved.
As is usual in the litterature, we call these methods linearly implicit.
They rely on the combination of a classical collocation Runge--Kutta method with a specific
treatment for the nonlinearity.

In particular, we show that, using the methods developed in this paper,
one can solve numerically virtually any ODE, up to any order,
by solving only linear systems at each time step.
Moreover, we believe that this new class of methods can help dramatically reducing the computational
time in several cases of time integration of evolution PDEs.
Indeed, after space discretization of an evolution PDE, if one uses, say,
an implicit (for stability reasons) Runge--Kutta method, then one needs to solve a nonlinear
problem in high dimension at each time step
(and one may use a fixed-point method or a Newton method to do so,
for example). With the methods introduced in this paper, the integration over any time step
can be carried out using only the solutions of linear systems in high dimension, and one can rely on
very efficient techniques, either
direct ($LU$ factorization, Choleski factorization, {\it etc}) or iterative (Jacobi method,
Gauss--Seidel method, conjugate gradient, Krylov subspace method, {\it etc} \cite{Saad03}),
depending on the structure of the problem at hand, to do so.

Of course, high order one step methods
in time exist in the literature since the pioneer work of Runge \cite{Runge1895}
and Kutta \cite{Kutta1901}.
The interested reader may refer to \cite{HNW93} for the integration of nonstiff problems and
\cite{HW96} for the integration of stiff problems, and to \cite{BW96} for historical notes
and references. Some methods are explicit, and lead,
for PDE problems, to restrictive CFL conditions in general. Some methods are fully implicit
and require the solution of nonlinear systems that are high dimensional in the PDE approximation
context. Other high order methods have been developed for PDEs. For example, high order exponential integrators have been used for parabolic problems \cite{HO05,HO10} and for NLS equations \cite{DUJ09,Nous2017}. Beyond the analysis presented in this paper in an ODE context, one of the goals
of this paper is to convince the reader that, in a PDE context, linearly implicit methods such
as that developed below can outperform classical methods from the literature with the same order.
This means that they require less CPU time to compute an approximation of the solution with
a given (small) error.

The methods introduced in this paper are {\it not} linear multistep methods (see \cite{Dahlquist63}
or Chapter III of \cite{HNW93}).
Indeed, for a general vector field, linear multistep methods are either explicit or fully
implicit, while the methods introduced in this paper are only linearly implicit.
Let us mention, however, that linearly implicit linear multistep methods have been developed
  and analysed, for example in \cite{Akrivis04} for nonlinear parabolic equations (see also
  \cite{Akrivis14}).
In this paper, we shall introduce suitable concepts of consistency, stability and convergence
for our methods, and we sometimes borrow the vocabulary to that of linear multistep methods,
but the definitions are indeed different. 
Note that the methods introduced in this paper are {\it not} one-step methods either.
Therefore, one cannot use composition techniques (see \cite{S90,YOSHIDA1990262,McLachlan2002})
directly to build up high order methods from lower order ones:
deriving a linearly implicit high order method (that is
convergent in a reasonable sense) is a challenge {\it per se}, that we tackle in this paper.
Another important class of high order methods, introduced by J. Butcher,
not to be confused with the one introduced in this paper, is that of DIMSIMs
(Diagonally Implicit MultiStep Integration Methods)
\cite{ButcherDIMSIM}, where the word "implicit" does not refer at all to "linearly implicit" but
rather to "fully implicit" (meaning : nonlinearly, even if diagonally). These methods
have been later generalised in a class called General Linear Methods (GLM) \cite{ButcherGLM},
that does not contain the methods introduced in this paper either.

The methods introduced in this paper are {\it not} classical linearly implicit
  methods either. Indeed, such methods have a long history, which
  dates back at least to the work of Rosenbrock \cite{Rosenbrock1963}.
  They have been developped and analysed on several evolution equations in several contexts
  by numerous authors. For example, the work of Rosenbrock was revisited in \cite{KW81} where
  Rosenbrock-Wanner (ROW) methods are introduced. The concept of $B$-convergence has been developed
  in \cite{Frank81}, analysed
  in \cite{Strehmel87,Strehmel88} for linearly implicit one step methods (see also \cite{Deuflhard83}).
  These methods have been applied to multibody
  systems in \cite{Wensch96}, to nonlinear parabolic equations in \cite{Lubich95},
  to advection-reaction-diffusion equations in \cite{Calvo01}
  and more recently to surface evolution in \cite{Kovacs18}.
  Indeed, such methods always involve derivatives of the vector field (or part of the vector field),
  or (sometimes crude) approximations to this derivative. In contrast, the methods developed
  in this paper do not : They rely on an accurate interpolation procedure in time
  for the nonlinear terms.

The methods introduced in this paper use additional variables to take care of the nonlinear terms
with high order accuracy, while making it possible to solve only linear systems at each time step.
Numerical methods using additional variables are not new. Indeed, such methods have been developped
in different contexts in order to achieve qualitative properties of the methods.
For example, the relaxation method introduced by C. Besse \cite{Besse2004} for the nonlinear Schr\"odinger (NLS) equation 
is a second order scheme \cite{BDDLV19} which preserves a discrete energy.
That relaxation method \cite{Besse2004} uses one additional variable to approximate part of
the nonlinearity in the NLS equation and is linearly implicit. In this sense,
the methods developped in this paper may be seen as generalizations of this
relaxation method. However, our goal is now to develop high order methods. To do so, we use
a higher order collocation approximation of part of the nonlinear terms in the equation.
Another class of numerical methods using additional variables is that of scalar auxiliary variable
methods (SAV) \cite{ShenXu18} and multiple scalar auxiliary variable (MSAV) \cite{ChengShen18}.
This class was introduced to produce unconditionnaly stable schemes for dissipative problems
with gradient flow structure. The auxiliary variable in this context is used to ensure
discrete energy decay. The order of the methods (1 or 2 in the references
above) is not the main issue.

Let us mention two additional goals that the authors aim at tackling with the methods introduced
in this paper.
First, the authors would like to be able to develop a stability
(and convergence) analysis for stiff problems,
{\it i.e. } an analysis with constants that depend only on the class of the linear
part of the vector field (later referred to as $L$, see \eqref{eq:PDE}) and not on
that linear part itself.
This would allow for the numerical treatment of evolution PDE problems, as well as their
space discretizations. This will be achieved in a forthcoming work, even if numerical
examples of PDE problems are presented in Section \ref{sec:num}.
Second, the authors would like to build up high order linearly implicit methods
with suitable qualitative properties ({\it e.g.} energy decay for dissipative problems,
or energy preservation for hamiltonian problems). For example, the relaxation method
\cite{Besse2004}, which belongs to the class of linearly implicit methods described in this paper,
preserves an energy when applied to the nonlinear Schr\"odinger equation.
For this reason, this paper only deals with constant time step methods.

The outline of this paper is as follows.
In Section \ref{sec:methstab}, we introduce the methods for a general
semilinear evolution problem (ODE or PDE) and
we introduce specific notions of stability, consistency and convergence for our class of methods.
Morover, we show, in a constructive way, that stable methods of arbitrarily high order exist in
Theorem \ref{th:specD} and Corollary \ref{cor:existence}. The main theoretical result of this paper
is that one can build up arbitrarily high order convergent linearly implicit methods
for ODEs (Theorem \ref{th:conv}).
We conclude Section \ref{sec:methstab} with examples of methods of order 1, 2, 4 and 6.
In Section \ref{sec:num}, we provide numerical examples of solutions of ODEs and PDEs.
These numerical experiments illustrate the convergence result of Theorem \ref{th:conv}
for evolution ODEs.
Moreover, they indicate that the result of Theorem \ref{th:conv} is still valid in several PDE
contexts.
We consider for example a NLS equation in 1d and 2d and
nonlinear heat equation in 1d.
The main result of the numerical experiments of Section \ref{sec:num} is that, for ODEs,
the linearly implicit methods do not dramatically outperform classical methods
from the literature with the same
order, no matter whether they are implicit or explicit (see Section \ref{subsec:scalODE}).
However, for the approximation of evolution PDEs in 1d (see Section \ref{subsubsec:Schro1d}),
with moderate space discretization, the linearly implicit methods show performances comparable to that
of explicit methods. Moreover, for the approximation of evolution PDEs in 2d
(see Section \ref{subsubsec:Schro2d}) with precise space discretization (leading to high number
of unknowns), the linearly implicit methods developed in this paper manage to outperform
standard methods from the literature with the same order.

\section{Linearly implicit methods of arbitrarily high order}
\label{sec:methstab}

\subsection{Introduction of the methods}
\label{subsec:IntroMeth}

We consider a semilinear autonomous evolution equation of the form
\begin{equation}
  \label{eq:PDE}
  \partial_t u = L u + N(u)u,
\end{equation}
where $L$ is a linear differential operator and $N$ is a nonlinear
function of $u$.
One can think for examples of the NLS equation, the nonlinear heat equation, or a simple ODE (see Section \ref{sec:num}
for actual examples).
We start at time $t=0$ with an initial datum $u^0$ (with $0$ in superscript)
in some functional space so that the Cauchy problem is well-posed on some interval $[0,T^\star)$
with $T^\star>0$. We choose $h>0$ and set $t_n=nh$ for $n\in\N$ as
long as $t_n<T^\star$.

Let us now start with the presentation of the new class of methods.
Assume a collocation Runge--Kutta method with $s\geq 1$ stages is given with coefficients
$0\leq c_1<\dots<c_s\leq 1$, $(a_{i,j})_{1\leq i,j\leq s}$ and $(b_i)_{1\leq i\leq s}$.
We denote by $c$ the vector $(c_i)_{1\leq i\leq s}$ and by $\mathds{1}$ the vector of size $s$ with all entries equal to one.  

We denote by $u$ the exact solution of \eqref{eq:PDE} and we set $\gamma(t)=N(u(t,\cdot))$.
We assume we are given $s$ approximations
\begin{equation*}
  \gamma_{n-1+c_i}\sim \gamma(t_{n-1}+c_i h) \qquad 1\leq i \leq s,
\end{equation*}
and another approximation $u_n\sim u(t_n,\cdot)$.
For possible ways of computing the $s$ first approximations, we refer to Remark \ref{rem:initgamma}.
For the numerical initial datum $u_0$ (with subscript $0$),
we consider an approximation of the exact initial datum $u^0$ (with superscript $0$) of equation
\eqref{eq:PDE} or its exact value.
For $(\theta_1,\dots,\theta_s)\in\R^s$ and $D\in{\mathcal M}_s(\R)$
to be chosen later, we define explicitly $(\gamma_{n+c_1},\dots,\gamma_{n+c_s})$
with the relation
\begin{equation}
\label{eq:heritagegamma}
  \left[
    \begin{matrix}
      \gamma_{n+c_1}\\
      \vdots\\
      \gamma_{n+c_s}
    \end{matrix}
  \right]
=
D
 \left[
    \begin{matrix}
      \gamma_{n-1+c_1}\\
      \vdots\\
      \gamma_{n-1+c_s}
    \end{matrix}
  \right]
+
 \left[
    \begin{matrix}
      \theta_{1}\\
      \vdots\\
      \theta_{s}
    \end{matrix}
  \right]
  N(u_n).
\end{equation}
Then, we define, linearly implicitly $(u_{n,1},\dots,u_{n,s})$ as the solution
of the Runge--Kutta like system
\begin{equation}
  \label{eq:RKimpl}
  u_{n,i}=u_n+h \sum_{j=1}^s a_{i,j}(L+\gamma_{n+c_j})u_{n,j}\qquad 1\leq i\leq s.
\end{equation}
Last, we set explicitly
\begin{equation}
  \label{eq:RKexpl}
  u_{n+1} = u_n+h \sum_{i=1}^s b_i(L+\gamma_{n+c_i})u_{n,i}.
\end{equation}

\noindent
The steps \eqref{eq:heritagegamma},\eqref{eq:RKimpl} and \eqref{eq:RKexpl} define
a linearly implicit method
\begin{equation*}
  (u_{n+1},\gamma_{n+c_1},\dots,\gamma_{n+c_s})=\Phi_{h}(u_{n},\gamma_{n-1+c_1},\dots,\gamma_{n-1+c_s}).
\end{equation*}

\begin{remark}
  \label{rem:relaxBesse}
  The relaxation method introduced for the NLS equation
  \begin{equation*}
    i\partial_t u + \Delta u = \lambda |u|^{2} u,
  \end{equation*}
  in \cite{Besse2004} writes
  \begin{equation}
    \label{eq:keepcoolBesse}
    \left\{
      \begin{array}{rcl}
    \dfrac{\phi_{n+1/2}+\phi_{n-1/2}}{2} & = & |u_n|^{2},\\
    i\dfrac{u_{n+1}-u_n}{h} & = & \left(-\Delta +\lambda \phi_{n+1/2}\right) \dfrac{u_n+u_{n+1}}{2}
    \end{array}
  \right.
  .
  \end{equation}
  This corresponds to taking $L=i\Delta$ and $N(u)=-i \lambda |u|^{2}$
  in \eqref{eq:PDE} and $s=1$, $a_{1,1}=\frac12$, $b_1=1$, $c_1=\frac12$, $\gamma_{n+1/2}=-i\lambda\phi_{n+1/2}$,
  $D=[-1]$ and $\theta_1=2$ in the numerical method
  \eqref{eq:heritagegamma},\eqref{eq:RKimpl},\eqref{eq:RKexpl}.

  In order to achieve order 2 with the relaxation method, C. Besse introduced a single auxiliary
  unknown $\phi$ on a staggered grid, corresponding to the relation $\phi=|u|^2$.
  Since we want to achieve higher orders, we decide to introduce several
  auxiliary unknowns on a staggered grid with $s$ points, corresponding to the
  relation $\gamma=N(u)$.

  Note that the convergence of order 2 of the relaxation method for the NLS equation
  is a difficult result and is {\it not} a consequence of the results of this paper, which
  only deals with ODEs in the theoretical part (Section 2) and allows for PDEs for illustration
  purposes (Section 3). Indeed, proving the convergence of a numerical time integration method
  applied to a PDE requires a functional analysis framework adapted to the PDE at hand and cannot
  in general
  be done once and for all. For the relaxation method applied to the NLS equation, the
  convergence of order 2 is proved in \cite{BDDLV19}.
  
\end{remark}

\subsection{Consistency and stability of the step \eqref{eq:heritagegamma}}

Let us denote by $\rho(D)$ the spectral radius of the matrix $D$, {\it i.e.} the biggest
modulus of its complex eigenvalues.
In view of relation \eqref{eq:heritagegamma}, we decide to set the following definitions
for the stability and consistency of the step \eqref{eq:heritagegamma}.
\begin{defi} \label{def:stab}
  The step \eqref{eq:heritagegamma} is said to be {\it stable} if
  \begin{equation*}
    \sup_{n\in\N} \|D^n\| < +\infty,
  \end{equation*}
  for some norm on ${\mathcal M}_s(\R)$.
  The step \eqref{eq:heritagegamma} is said to be {\it strongly stable} if $\rho(D)<1$.
\end{defi}
\begin{remark}
  In the definition of the stability above,
  the boundedness of the sequence  $(D^n)_{n \geq 0}$ is independant of the norm
  chosen on ${\mathcal M}_s(\R)$.
  Moreover, it is equivalent to the fact that $\rho(D)<1$ or $\rho(D)=1$ with simple Jordan
  blocks for $D$ for all eigenvalues of modulus 1.
  In particular, if the step \eqref{eq:heritagegamma} is strongly stable, then it is stable.
  The converse is not true in general.
  For example, the classical relaxation method of Remark \ref{rem:relaxBesse} is stable but
not strongly stable.
\end{remark}

In order to define the consistency of step \eqref{eq:heritagegamma},
we introduce the $s \times s$ square matrices $V_c, V_{c-\mathds{1}}$ and $\Theta$ defined by 
\begin{itemize}
	\item for all $i\geq 1$  and  $j \geq 1$,   $ \quad(V_c^h)_{ij}=(c_ih)^{j-1}$, 
	\item for all $i\geq 1$  and  $j \geq 1$, $\quad(V_{c-\mathds{1}}^h)_{ij}=((c_i-1)h)^{j-1}$,
	\item for all $i \geq 1$  and  $j \geq 2$,  $\quad (\Theta)_{i1}=\theta_i, ~ (\Theta)_{ij}=0$. 
\end{itemize}

\begin{defi} \label{def:cons}
We say that the step \eqref{eq:heritagegamma} is consistent of order $s$ if for all $h>0$,
\begin{equation}
\label{eq:rel_Vc_D}
V_c^h=D V_{c-\mathds{1}}^h+\Theta.
\end{equation}
\end{defi}
This relation holds for all $h>0$ if and only if it holds for $h=1$, as we show below.
Indeed, introducing the diagonal matrix $G(h)$ with coefficients $1, h, h^2, \dots, h^{s-1}$, one has 
$V_c^h=V_c^1G(h)$ and $V_{c-\mathds{1}}^h=V_{c-\mathds{1}}^1G(h)$. Since the $c_i$ are distinct,
the Vandermonde matrices $V_c^h$ and $V_{c-\mathds{1}}^h$ are invertible. By the relation \eqref{eq:rel_Vc_D} and using the matrix $G(h)$, we have
\begin{equation}
  \label{eq:relD}
  D=(V_c^h-\Theta)(V_{c-\mathds{1}}^h)^{-1}=(V_c^1G(h)-\Theta)G(h)^{-1}(V_{c-\mathds{1}}^1)^{-1}=V_c^1(V_{c-\mathds{1}}^1)^{-1}-\Theta G(h)^{-1}(V_{c-\mathds{1}}^1)^{-1}.
\end{equation}
Since $\Theta$ is zero except maybe on its first column, we have $\Theta (G(h))^{-1}=\Theta$.

The definition of the step \eqref{eq:heritagegamma} of the method
\eqref{eq:heritagegamma}-\eqref{eq:RKexpl} depends on the $s^2$ coefficients of the matrix
$D$ and the $s$ coefficients $\theta_1,\dots,\theta_s$.
Requiring that the step \eqref{eq:heritagegamma} is of order $s$ provides us with $s^2$ linear
equations between these unknowns (see relation \eqref{eq:relD}).
In Theorem \ref{th:specD}, we prove that we can add $s$ equations involving these unknows
by imposing the spectrum of the matrix $D$ and that the system that
we obtain has indeed a unique solution.
This will allow in particular to prove the existence of stable and strongly stable steps
\eqref{eq:heritagegamma} with order $s$ (see Corollary \ref{cor:existence}).

\begin{theorem}
  \label{th:specD}
  Assume $c_1,\dots,c_s$ are fixed and distinct as above.
  For all disctinct $\lambda_1,\dots,\lambda_s\in \C\setminus\{1\}$, there exists
  a unique $((\theta_1,\dots,\theta_s),D)\in \C^s\times{\mathcal M_s(\C)}$
  such that the step \eqref{eq:heritagegamma} is of order $s$ and the spectrum of the matrix $D$
  is exactly $\{\lambda_1,\dots,\lambda_s\}$.
  If moreover the set $\{\lambda_1,\dots,\lambda_s\}$ is stable under
  complex conjugation, then $(\theta_i)_{1\leq i\leq s}\in\R^s$
    and $D\in{\mathcal M}_s(\R)$.
\end{theorem}

\begin{proof}
  We set $M=(V_{c-\mathds{1}}^1)^{-1}V_c^1$.
  Note that $M$ is in fact independant of the choice of the $(c_i)_{1\leq i\leq s}$.
  Indeed, it is the matrix of the linear mapping $P(X)\mapsto P(X+1)$
  in the canonical basis of $\R_{s-1}[X]$. This means that the coefficients $(M_{ij})_{1\leq i,j\leq s}$
  of $M$ are given by $M_{ij}=0$ if $j<i$ and $M_{ij}={j-1\choose i-1}$ otherwise.
  In particular, it is upper triangular and its diagonal
  elements are equal to $1$.
  Assuming step \eqref{eq:heritagegamma} is consistent of order $s$, with \eqref{eq:relD}, we obtain
  \begin{equation}
    \label{eq:Dinterm}
  D = V_{c-\mathds{1}}^1\Big[(V_{c-\mathds{1}}^{1})^{-1}V_c^1-(V_{c-\mathds{1}}^1)^{-1}\Theta\Big] (V_{c-\mathds{1}}^1)^{-1},
\end{equation}
\noindent
so that the matrix $D$ is similar to $M-Y$, where $Y$ is the matrix $(V_{c-\mathds{1}}^1)^{-1}\Theta$.
Note that all the coefficients of $Y$ are equal to $0$, except maybe on the first column.
We shall denote by $y_1,\dots,y_s$ the coefficients in the first column of $Y$ and by $Y_1$
the first column of $Y$.
Given distinct $\lambda_1,\dots,\lambda_s\in\C\setminus\{1\}$,
the existence and uniqueness of $D$ and $\Theta$ such
that step \eqref{eq:heritagegamma} has order $s$ and the spectrum of $D$ is exactly
$\{\lambda_1,\dots,\lambda_s\}$ is equivalent to the existence and uniqueness of
$y_1,\dots,y_s\in\C$ such that $M-Y$ has spectrum $\{\lambda_1,\dots,\lambda_s\}$.

Let us fix $k\in\{1,\dots,s\}$. The existence of an eigenvector for $M-Y$ for the eigenvalue
$\lambda_k$ is exactly the existence of a nontrivial vector $Z_k\in\C^s$ such that
$(M-Y)Z_k=\lambda_k Z_k$. Let us denote by $I$ the identity matrix of size $s$ and $U$ the upper
triangular matrix such that $M=I-U$.
Observe that, in view of the definition of $M$,
  the entries above the diagonal of $U$ are negative.
The relation $(M-Y)Z_k=\lambda_k Z_k$ is equivalent to
$(1-\lambda_k) Z_k = (Y+U)Z_k$. Since $Y Z_k=z_1^{(k)}Y_1$ where $z_1^{(k)}$ is the first
component of $Z_k$, we infer that the relation $(M-Y)Z_k=\lambda_k Z_k$ is also equivalent to
\begin{equation}
  \label{eq:triangle}
  (1-\lambda_k) Z_k = z_1^{(k)}
  \left[
      \begin{matrix}
        y_1\\
        \vdots\\
        y_s
      \end{matrix}
    \right]
  +U Z_k.
\end{equation}
If $z_1^{(k)}=0$, then, because the matrix $U$ is strictly upper triangular and
$\lambda_k\neq 1$,
we have $Z_k=0$ and hence $Z_k$ is not an eigenvector.
Therefore, if $Z_k$ is an eigenvector, $z_1^{(k)}\neq 0$ and we can impose without loss of generality
that $z_1^{(k)}=1$. Together with \eqref{eq:triangle}, this implies, by recursion, that
\begin{equation}
  \label{eq:nilp}
  Z_k = \left( \sum_{p=0}^{s-1} \frac{1}{(1-\lambda_k)^{p+1}} U^p \right)
  \left[
      \begin{matrix}
        y_1\\
        \vdots\\
        y_s
      \end{matrix}
    \right].
\end{equation}
The equation on the first line in \eqref{eq:nilp} is of the form
\begin{equation}
  \label{eq:1ereligne}
  1 = P_1\left(\frac{1}{1-\lambda_k}\right) y_1+ \cdots + P_s\left(\frac{1}{1-\lambda_k}\right) y_s,
\end{equation}
where for all $i\in\{1,\dots,s\}$, $P_i$ is a polynomial of degree exactly $i$ (remind
that the matrix $U$ is strictly upper triangular with negative entries
above the diagonal).
Moreover, if the relation \eqref{eq:1ereligne} is verified for some $(y_1,\dots,y_s)$,
then there exists a solution $Z_k$ of \eqref{eq:triangle} with $z_1^{(k)}=1$:
one just has to compute the components of $Z_k$ in \eqref{eq:nilp} one after the other to obtain
an eigenvector of $M-Y$ for the eigenvalue $\lambda_k$.
As a summary, we have proved that, for all $k\in\{1,\dots,s\}$, $\lambda_k$ is an eigenvalue
of $M-Y$ if and only if \eqref{eq:1ereligne} holds.
Therefore, the fact that the spectrum of $M-Y$ is $\{\lambda_1,\cdots,\lambda_s\}$ is
equivalent to the linear system
\begin{equation}
  \label{eq:calcY}
  \left[
    \begin{matrix}
      P_1\left(\frac{1}{1-\lambda_1}\right) & \dots & P_s\left(\frac{1}{1-\lambda_1}\right)\\
      \vdots & & \vdots\\
      P_1\left(\frac{1}{1-\lambda_s}\right) & \dots & P_s\left(\frac{1}{1-\lambda_s}\right)
    \end{matrix}
  \right]
  \left[
      \begin{matrix}
        y_1\\
        \vdots\\
        y_s
      \end{matrix}
    \right]
    = {\mathds 1}.
\end{equation}
Since for all $i\in\{1,\dots,s\}$, the polynomial $P_i$ has degree $i$ and the $(\lambda_j)_{1\leq j\leq s}\in \left(\C\setminus\{1\}\right)^s$ are distinct, the system above is invertible, so that
it has a unique solution.
Since \eqref{eq:calcY} has a unique solution in $\C^s$ and the polynomials $(P_i)_{1\leq i\leq s}$
have real coefficients, it is easy to check that,
if the set $\{\lambda_1,\dots,\lambda_s\}$ is moreover stable under
complex conjugation, then $y_1,\dots,y_s$ are real numbers, and so are $\theta_1,\dots,\theta_s$
and the matrix $D$ has real coefficients using \eqref{eq:Dinterm}.
  This proves the theorem.
\end{proof}

\begin{corollary}
  \label{cor:existence}
  Assume $c_1,\dots,c_s$ are fixed and distinct as above. Then
  \begin{itemize}
    \item There exists $D\in{\mathcal M}_s(\R)$ and $\theta_1,\dots,\theta_s\in\R$ such that
      the step \eqref{eq:heritagegamma} is stable and has order $s$.
    \item There exists $D\in{\mathcal M}_s(\R)$ and $\theta_1,\dots,\theta_s\in\R$ such that
    the step \eqref{eq:heritagegamma} is strongly stable and has order $s$.
   \end{itemize}
 \end{corollary}

 \begin{proof}
   Choose $\lambda_1,\dots,\lambda_s\in\C\setminus\{1\}$ distinct such that for all $i$,
   $|\lambda_i|\leq 1$ to obtain a stable method (or for all $i$, $|\lambda_i|< 1$
   to obtain a strongly stable method) in such a way that the set $\{\lambda_1,\dots,\lambda_s\}$
   is stable under complex conjugation and apply Theorem \ref{th:specD}.
 \end{proof}

 \begin{remark}
   In order to actually build the matrix $D$ and the coefficients $\theta_1,\dots,\theta_s$
   that define step \eqref{eq:heritagegamma} so that this step is stable (respectively strongly stable)
   and has
   order $s$, it is sufficient to fix distinct $c_1,\dots,c_s$ as above, and choose
   distinct
   $\lambda_1,\dots,\lambda_s\in\C\setminus\{1\}$ with modulus less (respectively strictly less)
   than $1$,
   and in such a way that the set $\{\lambda_1,\dots,\lambda_s\}$
   is stable under complex conjugation.
   Then, one forms system \eqref{eq:calcY}, the rows of which are the first rows of the
   right-hand side of \eqref{eq:nilp} for different values of $\lambda_k$, and one solves
   it for $y_1,\dots,y_s$. One easily computes $\Theta$ from $Y$ using
   the fact that $\Theta= V_{c-\mathds{1}}^1 Y$. In the end, the matrix $D$ is computed
   using \eqref{eq:relD}.
   Examples are provided in Section \ref{subsec:examples}.
 \end{remark}

 \subsection{Convergence of the method \eqref{eq:heritagegamma}-\eqref{eq:RKexpl}}
 \label{subsec:CV}
 
 In this section, we prove that the methods presented above, provided that they involve
 a step \eqref{eq:heritagegamma} with strong stability and order $s$, and a Runge--Kutta collocation method of order at least $s$, are indeed convergent in
 finite time, with order $s$, when applied to an ODE with sufficiently smooth vector field.
We assume the unknown $u$ of equation \eqref{eq:PDE} is scalar and that $L=0$.  In fact, up to a change of unknown, any ODE with an equilibrium can be cast into this form:
 \begin{equation}
   \label{eq:PDEwithL0}
   u'(t)=N(u(t))u(t).
 \end{equation}
Our methods and results extend to systems of ODEs of the form \eqref{eq:PDEwithL0} where the unknown $u$ is vector-valued and $N$ is a given smooth matrix-valued function. Similarly, our methods and results extend to the case of complex-valued functions. But for the sake of simplicity we focus on the real valued scalar case. 

\bigskip
 
We assume $N$ is defined and smooth on some open subset $\Omega$ of $\R$. We fix $u^0\in \Omega$. There exists a unique maximal solution to the Cauchy problem
 \eqref{eq:PDEwithL0} for $u(0)=u^0$. This solution is defined on an open interval of the form
 $(T_\star,T^\star)$ with $-\infty\leq T_\star<0<T^\star\leq +\infty$. We fix $T\in (0,T^\star)$.
 Since the maximal solution is smooth, we have $\sup_{t\in[0,T]}|N(u(t))|<+\infty$.
 Since it is defined on the compact interval $[0,T]$, one can choose $r>0$ such that
 \begin{equation}
 \label{eq:defr}
 V=\{u(t)+v\ |\ t\in[0,T], \ v\in \R,\ |v|\leq r\}\subset\Omega.
 \end{equation}
 We set $M=\sup_{t\in [0,T]} |N(u(t))|$ and we choose $m>0$ such that $M+m>\sup_{u\in V}|N(u)|$.

 We discretize the time as in Section \ref{subsec:IntroMeth}, with $h$ small enough to ensure
 that $T_\star<t_{-1}$.
 We start by focusing on the consistency of the method.
 Namely, we set for all $n\in\N$ such that $t_n\leq T$,
 \begin{equation}
   \label{eq:Rn1}
   R_n^1=\left [
       \begin{matrix}
         N(u(t_n+c_1 h)) \\
         \vdots\\
         N(u(t_n+c_s h))
       \end{matrix}
     \right ]
   -
   D
   \left [
   \begin{matrix}
         N(u(t_{n-1}+c_1 h)) \\
         \vdots\\
         N(u(t_{n-1}+c_s h))
       \end{matrix}
     \right ]
   -
   N(u(t_n))
   \left [
       \begin{matrix}
       \theta_1\\
       \vdots\\
       \theta_s
       \end{matrix}
     \right ].
 \end{equation}
 Similarly, we define $R_n^2$ as the vector of $\R^s$ with entry number $i$ equal to
 \begin{equation}
   \label{eq:Rn2}
   \left(R_n^2\right)_i=u(t_n+c_ih)-u(t_n)-h\sum_{j=1}^s a_{ij} N(u(t_n+c_j h))u(t_n+c_j h),
 \end{equation}
 and
  \begin{equation}
   \label{eq:Rn3}
  R_n^3=u(t_{n+1})-u(t_n)-h\sum_{i=1}^s b_i N(u(t_n+c_i h))u(t_n+c_i h).
 \end{equation}
 
 \begin{lemma} \label{lem:Bertrand}
   Assume that the function $N$ is sufficiently smooth, $u^0 \in \Omega$ and $T \in (0,T^\star)$. Suppose moreover that the numerical coefficients $(c_i)_{1\leq i \leq s}$, $(a_{i,j})_{1 \leq i,j \leq s}$ and $(b_i)_{1\leq i \leq s}$ define a Runge--Kutta collocation method of order $s$ and that the step \eqref{eq:heritagegamma} is of order $s$. For any norm on $\R^s$, there exists a constant $C>0$ such that for a sufficiently small $h>0$,
   \begin{equation}
   \label{eq:est_R1}
     \max_{n\geq 0, \ t_{n+1}\leq T} \| R_n^1\| \leq C h^s,
   \end{equation}
   \begin{equation}
    \label{eq:est_R2}
     \max_{n\geq 0,\ t_{n+1}\leq T} \| R_n^2\| \leq C h^{s+1},
   \end{equation}
   \begin{equation}
    \label{eq:est_R3}
     \max_{n\geq 0, \ t_{n+1}\leq T} | R_n^3| \leq C h^{s+1}.
   \end{equation}
 \end{lemma}
\begin{proof} Let us start with the estimate on $R_n^1$. First of all we use a Taylor expansion and write for all $1 \leq i \leq s$
\begin{equation*}
N(u(t_n+c_ih))=\sum_{k=0}^{s-1} \dfrac{(c_ih)^k}{k!}(N \circ u)^{(k)}(t_n)+\int_{0}^{c_ih} \dfrac{(c_ih-\sigma)^{s-1}}{(s-1)!}(N \circ u)^{(s)}(t_n+\sigma) \rm d \sigma.
\end{equation*}
Let us denote by $X(t_n)$ the vector of $\R^s$ with $(N \circ u)^{(k-1)}(t_n)/(k-1)!$ as component number $k$. The relation above allows to write 
\begin{equation}
\label{eq:R1t1}
\left [
       \begin{matrix}
         N(u(t_n+c_1 h)) \\
         \vdots\\
         N(u(t_n+c_s h))
       \end{matrix}
     \right ]
= V_c X(t_n)
+
\left [
       \begin{matrix}
          \displaystyle\int_{0}^{c_1h} \frac{(c_1h-\sigma)^{s-1}}{(s-1)!}(N \circ u)^{(s)}(t_n+\sigma) \rm d \sigma\\
         \vdots\\
         \displaystyle\int_{0}^{c_sh} \frac{(c_sh-\sigma)^{s-1}}{(s-1)!}(N \circ u)^{(s)}(t_n+\sigma) \rm d \sigma
       \end{matrix}
     \right ]:=V_c X(t_n)+r_{1,1}^n.
\end{equation}
Similarly, we have 
\begin{eqnarray}
\left [
       \begin{matrix}
         N(u(t_{n-1}+c_1 h)) \\
         \vdots\\
         N(u(t_{n-1}+c_s h))
       \end{matrix}
     \right ]
& = & V_{c-\mathds{1}} X(t_n)
+
\left [
       \begin{matrix}
          \displaystyle\int_{0}^{(c_1-1)h} \frac{((c_1-1)h-\sigma)^{s-1}}{(s-1)!}(N \circ u)^{(s)}(t_n+\sigma) \rm d \sigma\\
         \vdots\\
         \displaystyle\int_{0}^{(c_s-1)h} \frac{((c_s-1)h-\sigma)^{s-1}}{(s-1)!}(N \circ u)^{(s)}(t_n+\sigma) \rm d \sigma
       \end{matrix}
  \right ]\nonumber\\ \label{eq:R1t2}
  & := & V_{c-\mathds{1}} X(t_n)+r_{1,2}^n.
\end{eqnarray}
Moreover we have
\begin{equation}
\label{eq:R1t3}
N(u(t_n))
   \left [
       \begin{matrix}
       \theta_1\\
       \vdots\\
       \theta_s
       \end{matrix}
     \right ]
=
\Theta X(t_n).
\end{equation}
Mutiplying \eqref{eq:R1t2} by $D$ and substracting the result and \eqref{eq:R1t3} to \eqref{eq:R1t1}, we infer that 
\begin{equation*}
R_n^1= V_c X(t_n)-D V_{c-1}X(t_n)-\Theta X(t_n) 
+ r_{1,1}^n- D r_{1,2}^n.
\end{equation*}  
Since the step \eqref{eq:heritagegamma} is of order $s$, we have using \eqref{eq:rel_Vc_D} 
\begin{equation*}
R_n^1= r_{1,1}^n- D r_{1,2}^n.
\end{equation*}  
Moreover with \eqref{eq:relD} we have 
$$D=V_c^1(V_{c-\mathds{1}}^1)^{-1}-\Theta (V_{c-\mathds{1}}^1)^{-1},$$
so that $D$ does not depend on $h$. Let $ \| \cdot  \|$ be a norm on $\R^s$. The vectors $r_{1,1}^n$ and $r_{1,2}^n$ satisfy
$$\max_{n\geq 0, \ t_{n+1}\leq T} \left(\| r^n_{1,1}\|+\| r^n_{1,2}\|\right) \leq C h^s,$$ for some $C>0$ and all sufficiently small $h>0$. This proves \eqref{eq:est_R1}. 
Since the Runge--Kutta method with coefficients $a_{i,j}$ and $b_i$ is a collocation method of order at least $s$ at points $c_i$, the bounds  \eqref{eq:est_R2} and \eqref{eq:est_R3} are classical (see for example Section II.1.2 in \cite{GNI2002}). 
\end{proof}

\begin{theorem} \label{th:conv}
Assume that the function $N$ is sufficiently smooth, $u^0 \in \Omega$ and $T \in (0,T^\star)$. Suppose moreover that the numerical coefficients $(c_i)_{1\leq i \leq s}$, $(a_{i,j})_{1 \leq i,j \leq s}$ and $(b_i)_{1\leq i \leq s}$ define a Runge--Kutta collocation method of order $s$ and that the step \eqref{eq:heritagegamma} is strongly stable and of order $s$.
Provided that $r$ is fixed as in \eqref{eq:defr} and $M$ and $m$ accordingly, there exists constants $C>0$ and $h_0>0$, such that, for all $h \in (0,h_0)$, if the initial data $u_0 \in \Omega$ (respectively $(\gamma_{-1+c_1},\cdots,\gamma_{-1+c_s}) \in N(V)^s$) is sufficiently close to its exact analogues $u^0 \in \Omega$ (respectively $(N(u(t_{-1}+c_1h)), \cdots, N(u(t_{-1}+c_sh))) \in N(V)^s$) in the sense of relations \eqref{eq:contrainteu0}--\eqref{eq:contraintegammainit}, then for all $n \in \N$ such that $t_{n-1} \leq T$, the step \eqref{eq:RKimpl} has a unique solution in $\R^s$ and for all $n \in \N$ such that $t_{n} \leq T$,
\begin{eqnarray}
&\left|\gamma_{n-1+c_i}\right| \leq M+m, \qquad \forall ~ i  \in \ldbrack 1,s\rdbrack&\label{eq:estconvgamma}\\
&\displaystyle{\max_{k \in \ldbrack 0, n \rdbrack} |u(t_k)-u_k| \leq e^{Cnh} \left[|u^0-u_0| + C \left(\max_{i \in \ldbrack 1,s \rdbrack} |\gamma_{-1+c_i} - N(u(t_{-1}+c_ih))| + h^s\right)\right].}&  \label{estconvu}
\end{eqnarray}

\end{theorem}

Let us first introduce all the notations we use in the proof.
We denote by $\Gamma_n$ the vector of $\R^s$ with component $i$ equal to $\gamma_{n+c_i}$. Let us define the convergence errors $P_n \in \R^s$ with component number $i$ equal to $(P_n)_i=N(u(t_n+c_ih))-\gamma_{n+c_i}$, $Q_n \in \R^s$ with component number $i$ equal to $Q_{n,i}=u(t_n+c_ih)-u_{n,i}$ (provided $u_{n,i}$ is well defined), and $e_n \in \R$ with $e_n=u(t_n)-u_n$. We set $z_n=\max_{0 \leq k \leq n} |e_k|$.  We denote by $|\cdot|_\infty$ the norm on $\R^s$ defined as the maximum of the absolute values of the components of the vectors. Moreover, we denote by $\| \cdot\|_\infty$ the norm on $\mathcal{M}_s(\R)$ induced by $|\cdot|_\infty$. In the following proof, the letter $C$ denotes a positive real number which does not depend on $h$ (but depends on $M$ and $r$ in particular) and whose value may vary from one line to the other. 

\begin{proof}
Since step \eqref{eq:heritagegamma} is strongly stable we have $\rho(D) <1$. Therefore, there exists a norm $| \cdot |_D$ on $\R^s$ such that the norm $\| \cdot \|_D$ induced by this norm on $\mathcal{M}_s(\R)$ satisfies  $\| D \|_D < 1$. In the following we set $\delta=\| D \|_D$. Since $\R^s$ is of dimension $s$, there exists a $\kappa \in (0,1]$ such that for all $x$ in $\R^s$, $\kappa |x|_D \leq  |x|_\infty \leq \dfrac{1}{\kappa}  |x|_D$. 

We divide the proof in two parts. First we assume an a priori bound for the numerical solution. Namely we assume that for all $n$ such that $t_n \leq T$:
\begin{itemize}
	\item (H1) $|\Gamma_n|_\infty \leq M+m$,
	\item (H2) the step \eqref{eq:RKimpl} has a unique solution $(u_{n,i})_{1\leq i \leq s}$ in $\R^s$, 
	\item (H3) $u_n \in V$. 
\end{itemize}
We show that, in this case, we have an explicit bound for the convergence errors $P_n$ and $z_n$ (see equations \eqref{eq:recestznfin} and \eqref{eq:Pnderoule3}). 

Second, we assume that $h_0$ and the initial errors $P_{-1}$ and $e_0$ are small enough and we show that the bounds of the first part of the proof are indeed satisfied. 

\vskip0.2cm
\noindent{\bf First part.} In addition to the bounds above, we assume that $h\in (0,1)$ and $n$ satisfy $t_{n+1} \leq T$.  
Substracting \eqref{eq:heritagegamma} from \eqref{eq:Rn1} we obtain
\begin{equation}
\label{eq:recerrPn}
P_n=DP_{n-1}+ \left(N(u(t_n))-N(u_n)\right)\left [
       \begin{matrix}
       \theta_1\\
       \vdots\\
       \theta_s
       \end{matrix}
     \right ]
+ R_n^1.
\end{equation}
We infer that 
\begin{equation}
\label{eq:recestPn}
|P_n|_D \leq \|D \|_D |P_{n-1}|_D + C |e_n| + |R_n^1|_D,
\end{equation}
where the constant $C$ is the product of the Lipschitz constant of $N$ over the compact $V$ times $|(\theta_1, \cdots, \theta_s)^t|_D$ (recall that $N$ is a smooth function over the open set
$\Omega$, hence it is Lipschitz-continuous on the compact $V\subset\Omega$). 

Substracting \eqref{eq:RKimpl} from \eqref{eq:Rn2} we obtain
\begin{eqnarray*}
\lefteqn{Q_{n,i}}\\
&=&e_n + h \sum_{j=1}^s a_{i,j} \left(N(u(t_n+c_jh))u(t_n+c_jh)-\gamma_{n+c_j}u_{n,j}\right) +(R_n^2)_i \\
&=& e_n + h \sum_{j=1}^s a_{i,j} \left(N(u(t_n+c_jh))-\gamma_{n+c_j}\right)u(t_n+c_jh) +h \sum_{j=1}^s a_{i,j} \gamma_{n+c_j}\left(u(t_n+c_jh)-u_{n,j}\right) +(R_n^2)_i \\
&=& e_n + h \sum_{j=1}^s a_{i,j} P_{n,j}u(t_n+c_jh) +h \sum_{j=1}^s a_{i,j} \gamma_{n+c_j}Q_{n,j} +(R_n^2)_i.
\end{eqnarray*}
We infer that
\begin{equation*}
|Q_n|_\infty \leq |e_n| + Ch |P_n|_\infty + Ch |\Gamma_n|_\infty |Q_n|_\infty + |R_n^2|_\infty , 
\end{equation*}
which gives with the first point of the assumptions above
\begin{equation*}
|Q_n|_\infty \leq |e_n| + Ch |P_n|_\infty + Ch(M+m) |Q_n|_\infty + |R_n^2|_\infty. 
\end{equation*}
Provided that $Ch(M+m) \leq 1/2$, we have
\begin{equation}
\label{eq:recestQn}
|Q_n|_\infty \leq 2|e_n| + Ch |P_n|_\infty+ 2|R_n^2|_\infty. 
\end{equation}

Substracting \eqref{eq:RKexpl} from \eqref{eq:Rn3} we obtain 
\begin{eqnarray*}
e_{n+1}&=& e_n + h \sum_{i=1}^s b_i \left(N(u(t_n+c_ih))-\gamma_{n+c_i}\right)u(t_n+c_ih) +h \sum_{i=1}^s b_i \gamma_{n+c_i}\left(u(t_n+c_ih)-u_{n,i}\right) +R_n^3 \\
&=& e_n + h \sum_{i=1}^s b_i P_{n,i}u(t_n+c_ih) +h \sum_{i=1}^s b_i \gamma_{n+c_i}Q_{n,i} +R_n^3. 
\end{eqnarray*}
We infer that
\begin{equation*}
|e_{n+1}| \leq |e_n| + Ch |P_n|_\infty + Ch |\Gamma_n|_\infty |Q_n|_\infty + |R_n^3| , 
\end{equation*}
which gives with the first point of the assumptions above
\begin{equation*}
|e_{n+1}| \leq |e_n| + Ch |P_n|_\infty + Ch(M+r) |Q_n|_\infty + |R_n^3|. 
\end{equation*}
Using \eqref{eq:recestQn}, we have
\begin{equation}
\label{eq:recesten}
|e_{n+1}| \leq (1+Ch)|e_n| + Ch |P_n|_\infty+ Ch |R_n^2|_\infty+ |R_n^3|. 
\end{equation}
From \eqref{eq:recestPn} we have by induction 
\begin{equation}
\label{eq:Pnderoule}
|P_n|_D \leq \delta^{n+1} |P_{-1}|_D + C \sum_{k=0}^n \delta^{n-k} \left(|e_k|+ |R_k^1|_D\right).
\end{equation}
Using the norm equivalence and \eqref{eq:Pnderoule} in \eqref{eq:recesten} we obtain
\begin{equation}
\label{eq:recesten2}
|e_{n+1}| \leq (1+Ch)|e_n| + \dfrac{Ch}{\kappa} \left[\delta^{n+1} |P_{-1}|_D + C \sum_{k=0}^n \delta^{n-k} \left(|e_k|+ |R_k^1|_D\right)\right]+ Ch |R_n^2|_\infty+ |R_n^3|, 
\end{equation}
which gives with Lemma \ref{lem:Bertrand}
\begin{equation}
\label{eq:recesten3}
|e_{n+1}| \leq (1+Ch)|e_n| + \dfrac{Ch}{\kappa} \left[\delta^{n+1} |P_{-1}|_D + C \sum_{k=0}^n \delta^{n-k} \left(|e_k|+ |R_k^1|_D\right)\right]+ Ch^{s+2}+ Ch^{s+1}.
\end{equation}
Using the maximal error defined previously and the fact that $\delta <1$ since the step \eqref{eq:heritagegamma} is strongly stable, we have
\begin{eqnarray*}
|e_{n+1}| &\leq& (1+Ch)z_n + Ch \left[\delta^{n+1} |P_{-1}|_D + \left(z_n+ h^s\right)\sum_{k=0}^n \delta^{n-k} \right]+ Ch^{s+1} \\
&\leq& (1+Ch)z_n + Ch \left[\delta^{n+1} |P_{-1}|_D + \left(z_n+ h^s\right)\dfrac{1}{1-\delta} \right]+ Ch^{s+1} ,
\end{eqnarray*}
and then 
\begin{equation}
\label{eq:recesten4}
|e_{n+1}| \leq (1+Ch)z_n + Ch \delta^{n+1} |P_{-1}|_D + Ch^{s+1}.
\end{equation}
Using that $z_{n+1}=\max \{z_n, |e_{n+1}|\}$ and $\delta\in (0,1)$, we infer
\begin{equation}
\label{eq:recesten5}
z_{n+1}\leq (1+Ch)z_n + Ch  |P_{-1}|_D + Ch^{s+1}.
\end{equation}
By induction it follows that for all $n$ in $\N$ such that $t_{n}\leq T$,
\begin{eqnarray}
z_{n}&\leq& (1+Ch)^n z_0 + Ch  \left(|P_{-1}|_D + h^{s}\right)\sum_{k=0}^{n-1} (1+Ch)^k  \nonumber \\
&\leq& e^{Cnh} z_0 + Ch  \left(|P_{-1}|_D + h^{s}\right)\dfrac{(1+Ch)^{n}}{1+Ch-1} \nonumber \\
&\leq& e^{Cnh}\left( z_0 + C \left(|P_{-1}|_D +h^{s}\right)\right)\nonumber \\
&\leq& e^{Cnh}\left( z_0 + C \left(|P_{-1}|_\infty +h^{s}\right)\right). \label{eq:recestznfin}
\end{eqnarray}
Using \eqref{eq:Pnderoule} and the same estimations as above, we have moreover
\begin{equation}
\label{eq:Pnderoule2}
|P_n|_D \leq |P_{-1}|_D + C(z_n+h^s).
\end{equation}
We infer
\begin{equation}
\label{eq:Pnderoule3}
|P_n|_\infty \leq Ce^{Cnh}\left(z_0+|P_{-1}|_\infty +h^s\right).
\end{equation}

\vskip0.2cm
\noindent{\bf Second part.} 
From now on, we denote by $C$ the maximum of the constants appearing in the right hand sides of \eqref{eq:recestznfin} and \eqref{eq:Pnderoule3}. Choose $h_0 \in (0,1)$ sufficiently small to have $h_0<\min\{-T_\star,T^\star\}$ and $Ce^{CT}h_0^s<r$ and $Ce^{CT}h_0^s<m$ and $h_0 \|A\|_\infty (M+m)<1$. Assume $u_0, \gamma_{-1+c_1}, \dots, \gamma_{-1+c_s} \in \R$ and $h \in (0,h_0)$ satisfy 
\begin{equation}
\label{eq:contrainteu0}
e^{CT}\left( |u^0-u_0| + C \left(\max_{i \in \ldbrack 1,s \rdbrack} |\gamma_{-1+c_i} - N(u(t_{-1}+c_ih))| +h_0^{s}\right)\right)<r,
\end{equation} 
and
\begin{equation}
\label{eq:contraintegammainit}
Ce^{CT}\left(|u^0-u_0|+\max_{i \in \ldbrack 1,s \rdbrack} |\gamma_{-1+c_i} - N(u(t_{-1}+c_ih))| +h_0^s\right)<m.
\end{equation} 

First, with \eqref{eq:Pnderoule3} and \eqref{eq:contraintegammainit}, we have 
$$|P_{0}|_\infty = \max_{i \in \ldbrack 1,s \rdbrack} |\gamma_{0+c_i} - N(u(t_{0}+c_ih))| < m.$$
Therefore by triangle inequality we have 
$$|\Gamma_{0}|_\infty \leq |P_{0}|_\infty + \left| (N(u(t_{0}+c_ih)))_{1\leq i \leq s} \right|_\infty<M+m.$$
And then, the hypothesis (H1) of the first part is satisfied for $n=0$.

Moreover, with \eqref{eq:contrainteu0}, we have $|u^0-u_0| \leq r$ so that $u_0 \in V$ and the hypothesis (H3) of the first part is satisfied with $n=0$. We infer that the system \eqref{eq:RKimpl} (with $L = 0$) has a unique solution in $\R^s$ since we assumed $h \leq h_0 < 1/(\|A\|_\infty(M+m))$. This implies that the hypothesis (H2) of the first part is satisfied for $n=0$. Then we can apply the analysis of the first part to obtain \eqref{eq:recestznfin} and \eqref{eq:Pnderoule3} with $n=1$. Using \eqref{eq:contrainteu0} and \eqref{eq:contraintegammainit}, we infer that hypotheses (H1), (H2) and (H3) are satisfied with $n=1$ and the result follows by induction on $n$.

\end{proof}

\begin{remark}
  The proof may look like following the usual strategy for the proof of convergence of a numerical
  method. However, note that the definition of strong stability (Definition \ref{def:stab}) plays
  a central role in the proof, and this is not the case in classical theory.
  Moreover, estimations such as \eqref{eq:Pnderoule} in
  \eqref{eq:recesten} to obtain \eqref{eq:recesten2} are {\it not} classical.
\end{remark}

\begin{remark}
  \label{rem:initgamma}
  Given $u_0\in\Omega$ close to $u^0\in\Omega$, before starting the time-stepping
  method \eqref{eq:RKimpl}-\eqref{eq:RKexpl}, one needs to first compute the $s$ approximations $(\gamma_{-1+c_i})_{1\leq i\leq s}$ of
  $(N(u((-1+c_i)h)))_{1\leq i\leq s}$. These quantities enter the error estimate \eqref{estconvu}.
  This computation can be done efficiently by determinating $s$ approximations of $(u((-1+c_i)h))_{1\leq i\leq s}$
  using a sufficiently high order method for \eqref{eq:PDE} backwards in time, provided the equation
  makes sense. Alternatively, for example for the nonlinear heat equation, for which
  running \eqref{eq:PDE} backwards in time makes no sense, one can use a sufficiently
  high order method from $u_0$ to compute $s+1$ approximations of $(u(c_i h))_{1\leq i\leq s}$
  and $u(h)$ over one time step, and then start the linearly implicit method
  \eqref{eq:RKimpl}-\eqref{eq:RKexpl} from time $h$ until time $T$.
\end{remark}

\subsection{Examples of linearly implicit methods} \label{subsec:examples}

In this section we present possible choices of methods of order 1, 2, 4 and 6. The general building procedure is the following: We choose $s \in \N^\star$, we fix $0 \leq c_1 <c_2< \dots <c_s \leq 1$ and we compute $a_{i,j}$ and $b_i$ for $1 \leq i,j \leq s$ using the formulas 
$$a_{i,j}=\int_0^{c_i} \mathcal{L}_j(\tau) {\rm d}\tau \qquad {\rm and} \qquad b_i = \int_0^{1} \mathcal{L}_i(\tau) {\rm d}\tau,$$
where $\mathcal{L}_i(\tau)=\displaystyle{\prod_{\substack{k=1 \\ k\neq i}}^s } \dfrac{\left(\tau-c_k\right)}{(c_i-c_k)}$ is the $i^{\rm th}$ Lagrange polynomial at points $c_1, \dots, c_s$. This way, the coefficients $a_{i,j}$, $b_i$ and $c_i$ are those of a Runge--Kutta collocation method. Next we choose $\lambda_1,\dots,\lambda_s \in \C \setminus \{1\}$ with moduli strictly less than 1, all distincts and in such a way that the set $\{\lambda_1,\dots,\lambda_s\}$ is invariant under complex conjugation. We compute the polynomials $P_1,\dots, P_s$ appearing in \eqref{eq:1ereligne} defined using \eqref{eq:nilp} in the proof of Theorem \ref{th:specD}. We solve \eqref{eq:calcY} for $y_1, \dots, y_s$ and compute $\theta_1, \dots, \theta_s$ using $\Theta=(V_{c-\mathds{1}}^1)Y$. Finally, we compute the matrix $D$ using \eqref{eq:relD}. This way, we define a step \eqref{eq:heritagegamma} that is strongly stable and of order $s$ (see Definitions \ref{def:stab} and  \ref{def:cons}). Using Theorem \ref{th:conv}, the numerical method \eqref{eq:heritagegamma}-\eqref{eq:RKexpl} is convergent of order $s$. 

\vskip0.2cm
\noindent \underline{\itshape A linearly implicit method of order 1:} We choose $s=1$ and $c_1=1$ so as to rely on the implicit Euler method. Then $a_{1,1}=1$ and $b_1=1$. Choosing $\lambda_1=1/2$, we have $y_1=1/2$ and $\theta_1=1/2$. 

\vskip0.2cm
\noindent \underline{\itshape Two linearly implicit methods of order 2:} 

\underline{\itshape 1- With Gauss points:} For $s=2$ and the Gauss points $c_1=\frac{1}{2}-\frac{\sqrt{3}}{6}, c_2=\frac{1}{2}+\frac{\sqrt{3}}{6}$. Then the Runge--Kutta collocation method has Butcher tableau
\[
\renewcommand\arraystretch{1.2}
\begin{array}
{c|cc}
 \frac{1}{2}-\frac{\sqrt{3}}{6} & \frac{1}{4} & \frac{1}{4}-\frac{\sqrt{3}}{6}\\
\frac{1}{2}+\frac{\sqrt{3}}{6} & \frac{1}{4}+\frac{\sqrt{3}}{6} & 1/4\\
\hline
& \frac{1}{2} & \frac{1}{2} 
\end{array}.
\]

\underline{\itshape 2- With uniform points:} For $s=2$ and the uniform points $c_1=0, c_2=1$. Then the Runge--Kutta collocation method has Butcher tableau
\[
\renewcommand\arraystretch{1.2}
\begin{array}
{c|cc}
0 & 0 & 0\\
1 & 1/2 & 1/2\\
\hline
& \frac{1}{2} & \frac{1}{2} 
\end{array}.
\]

For the two cases, we choose $\lambda_1=1/2, \lambda_2=-1/2$. This leads to $y_1=2, y_2=3/4$ and $\theta_1=y_1+(c_1-1)y_2, \theta_2=y_1+(c_2-1)y_2$.  

\vskip0.2cm
\noindent \underline{\itshape A linearly implicit method of order 4:} We choose $s=4$ and $c_1=0, c_2=1/3,c_3=2/3, c_4=1$. Then the Runge--Kutta collocation method has Butcher tableau
\[
\renewcommand\arraystretch{1.2}
\begin{array}
{c|cccc}
0 & 0 & 0 & 0 & 0\\
1/3& 1/8 & 19/72 & -5/72 & 1/72\\
2/3 & 1/9 & 4/9 & 1/9 & 0\\
1& 1/8 & 3/8 & 3/8 & 1/8\\
\hline
& 1/8 & 3/8 & 3/8 & 1/8 
\end{array}.
\]
Choosing $\lambda_1=0, \lambda_2=1/4, \lambda_3=1/2, \lambda_4=3/4$ for  we have 
$$\begin{pmatrix} 
y_1 \\ y_2 \\ y_3\\ y_4 \end{pmatrix} = \begin{pmatrix} 5/2 \\ 117/64 \\ 11/32 \\ 1/64 \end{pmatrix}\qquad {\rm and} \qquad \begin{pmatrix} 
\theta_1 \\ \theta_2 \\ \theta_3\\ \theta_4 \end{pmatrix} = \begin{pmatrix}  1 \\ 1235/864 \\ 833/432\\ 5/2 \end{pmatrix}.$$ 

\vskip0.2cm
\noindent \underline{\itshape A linearly implicit method of order 6:} We choose $s=6$ and $(c_i)_{1\leq i \leq 6}$ a uniform subdivision of $[0,1]$. Then the Runge--Kutta collocation method has Butcher tableau
\[
\renewcommand\arraystretch{1.2}
\begin{array}
{c|cccccc}
0 & 0 & 0 & 0 & 0 & 0 & 0\\
1/5& 19/288 & 1427/7200 & -133/1200 & 241/3600 & -173/7200 & 3/800\\
2/5 & 14/225 & 43/150 & 7/225 & 7/225 & -1/75 & 1/450 \\
3/5& 51/800 & 219/800 & 57/400 & 57/400 & -21/800 & 3/800\\
4/5 & 14/225 & 64/225 & 8/75 & 64/225 & 14/225 & 0 \\
1& 19/288 & 25/96 & 25/144 & 25/144 & 25/96 & 19/288\\
\hline
& 19/288 & 25/96 & 25/144 & 25/144 & 25/96 & 19/288 
\end{array}.
\]
Choosing $\lambda_k=\dfrac{e^{i(k-1)\frac{\pi}{3}}}{2}$ for $k=1,\dots,6$, we have 
$$\begin{pmatrix} 
y_1 \\ y_2 \\ y_3\\ y_4\\y_5 \\y_6 \end{pmatrix} = \begin{pmatrix} 6 \\ 2783/320 \\ 1239/256 \\ 659/512 \\ 43/256 \\ 21/2560 \end{pmatrix} \qquad {\rm and} \qquad \begin{pmatrix} 
\theta_1 \\ \theta_2 \\ \theta_3\\ \theta_4\\\theta_5 \\\theta_6 \end{pmatrix} = \begin{pmatrix} 65/64 \\ 193389/125000 \\ 1133667/500000 \\ 1608733/500000 \\ 1111047/250000 \\ 6 \end{pmatrix}.$$ 


\begin{remark}
  In the linearly implicit methods given as examples above, the choice of $(\lambda_i)_{1\leq i\leq s}$
  is somehow arbitrary, and the only condition we impose is that they ensure that
  step \eqref{eq:heritagegamma} is strongly stable (see Definition \ref{def:stab}).
  This implies that these methods are convergent for ODEs (see Theorem \ref{th:conv}).
  In order to ensure additional features of a linearly implicit method, this choice has to be made
  carefully (see for example in Section \ref{subsec:heat} a linearly implicit
  method that preserves non-negativity and energy-dissipation for a nonlinear heat equation).
  For a general class of evolution PDE, the choice of the collocation method as well
  as that of $(\lambda_i)_{1\leq i\leq s}$ has to be made carefully, in particular to ensure a taylored
  stability property of the linearly implicit method.
  This question will be investigated in a forthcoming paper.
\end{remark}

\section{Numerical experiments}
\label{sec:num}

In this section, we illustrate the properties of the methods described in Section
\ref{subsec:IntroMeth} and analysed in Section \ref{subsec:CV}.
We first present numerical examples on ODEs, with a scalar case in Section \ref{subsec:scalODE}.
In particular, we illustrate the results above, such as Theorem \ref{th:conv}, for several
methods introduced above, and we compare the results we obtain with that obtained using other
classical numerical methods.
Then, we present numerical experiments for PDEs that fit the framework used in Section
\ref{subsec:IntroMeth}
but do not fit {\it stricto sensu} the framework of the analysis carried out in Section
\ref{subsec:CV}. This allows for comparison with classical methods for the same problems anyway.
We first focus on a nonlinear Schr\"odinger equation in Section \ref{subsec:Schro}
and then move to a nonlinear heat equation in Section \ref{subsec:heat}.
The methods described and analysed in this paper would also be relevant
for several other examples of semilinear evolution equation of the form \eqref{eq:PDE}.

When comparing the efficiency of numerical methods in this section,
we consider as a measure of performance
the (lowest possible) CPU time required to achieve a given precision on the numerical
result. This CPU time has indeed disadvantages since it depends on the algorithms used
to solve the problems, the software used to implement the algorithms and the machine on which
the software is run. However, we believe one cannot talk about efficiency whithout taking into
account some form of CPU time. And, for reproducibility issues, we detail below as much as possible
which discretizations and algorithms are used to implement the numerical methods that we consider.
Moreover, we try to be as fair as possible when implementing methods from the litterature to compare
them with the linearly implicit methods introduced in this paper.

As we shall see in this section, the efficiency of the linearly implicit methods introduced
in this paper is similar to that of classical one-step methods with constant step size
from the literature (see Section \ref{subsec:scalODE}).
In contrast, the linealy implicit methods sometimes outperform standard methods
when applied to several evolution
PDE problems, that we consider, once discretized, as high dimensional systems of ODEs
(see Sections \ref{subsec:Schro} and \ref{subsec:heat}).
In the following, the computations are carried out using MATLAB and the linear systems are solved
using the backslash MATLAB command. In particular, we do not build a taylored
method to solve the linear systems numerically and the gain in computational time one can obtain using
linearly implicit methods can surely be improved using taylored methods depending on the matrix
structures. This choice is not optimal in terms of efficiency, but is fairly similarly done for
all the methods below.

\subsection{Application to a scalar nonlinear ODE}
\label{subsec:scalODE}
We consider the scalar ODE
\begin{equation}
\label{eq:EDO3}
  u'(t)=-u(t)-u^2(t).
\end{equation}

This corresponds to taking $L$ as minus the identity operator and $N(u)=-u$ in
\eqref{eq:PDE}.
The exact maximal solution starting from $u_0>0$ at $t=0$ is given for $t \geq 0$ by
\begin{equation*}
  u(t)=\frac{1}{(\frac{1}{u_0}+1){\rm e}^{t}-1}.
\end{equation*}

We start with methods of order 1.
We use the linearly implicit method of order 1 introduced in Section \ref{subsec:examples}.
We compare the results we obtain on the problem above with the Euler implicit and explicit schemes
as well as the Lie splitting method.
We choose $u^0=u_0=1/3$, $\gamma_{-1+c_1}=\gamma_0=N(u^0)$ and the final time $T=2$.
The results are displayed in Figure \ref{fig:ordre1_EDO}.
The global error is defined as $z_N$
  (with the notations of the proof of Theorem \ref{th:conv}) at final time $T$ with $N$ such that
  $Nh=T$, where $h$ is the time step.
Numerical experiments indicate that the four schemes are of order 1.
For the linearly implicit scheme, this is a consequence of Theorem \ref{th:conv}.
Moreover the CPU time required to reach a given numerical error is much lower for
the Lie splitting than for the linearly implicit method and for the linearly implicit method than for the explicit Euler scheme and the implicit Euler scheme. 

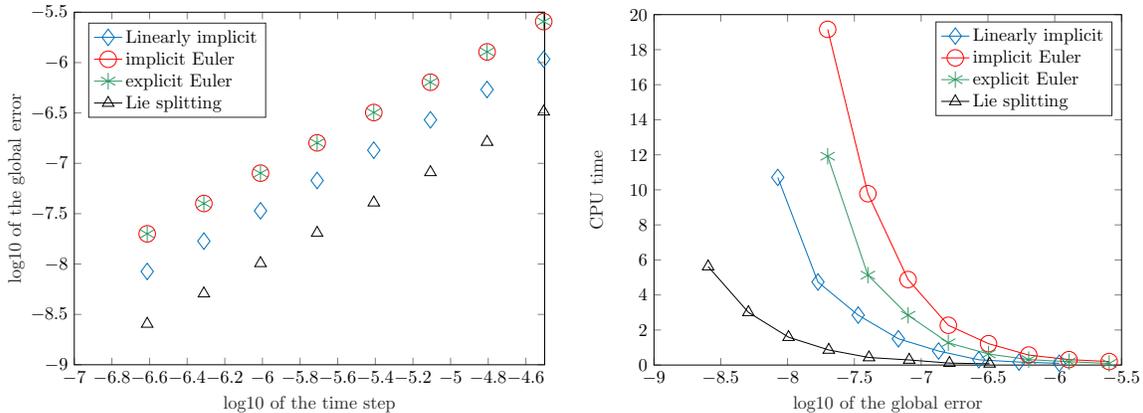
\begin{figure}[!h]
  \centering
   \begin{tabular}{cc}
\resizebox{0.48\textwidth}{!}{
%
%
\definecolor{mycolor1}{rgb}{0.00000,0.44700,0.74100}%
\definecolor{mycolor2}{rgb}{0.2,0.6,0.4}%
\begin{tikzpicture}

\begin{axis}[%
width=10cm,
height=7.5cm,
at={(0.758in,0.481in)},
scale only axis,
xmin=-7,
xmax=-4.5,
xlabel style={font=\color{white!15!black}},
xlabel={log10 of the time step},
ymin=-9,
ymax=-5.5,
ylabel style={font=\color{white!15!black}},
ylabel={log10 of the global error},
axis background/.style={fill=white},
legend style={at={(0.03,0.97)}, anchor=north west, legend cell align=left, align=left, draw=white!15!black}
]
\addplot [color=mycolor1, draw=none, mark size=5pt, mark=diamond, mark options={solid, mycolor1}]
  table[row sep=crcr]{%
-4.50514997831991	-5.96636311107174\\
-4.80617997398389	-6.26739428806687\\
-5.10720996964787	-6.56842486934128\\
-5.40823996531185	-6.86945515813429\\
-5.70926996097583	-7.17048527352206\\
-6.01029995663981	-7.4715153386196\\
-6.31132995230379	-7.77254549028603\\
-6.61235994796777	-8.07357554092571\\
};
\addlegendentry{Linearly implicit}

\addplot [color=red, draw=none, mark size=5.0pt, mark=o, mark options={solid, red}]
  table[row sep=crcr]{%
-4.50514997831991	-5.59266548544308\\
-4.80617997398389	-5.89369122568333\\
-5.10720996964787	-6.19471909047088\\
-5.40823996531185	-6.49574803382238\\
-5.70926996097583	-6.79677749193149\\
-6.01029995663981	-7.09780720185103\\
-6.31132995230379	-7.39883712202114\\
-6.61235994796777	-7.69986712885821\\
};
\addlegendentry{implicit Euler}

\addplot [color=mycolor2, draw=none, mark size=5.0pt, mark=asterisk, mark options={solid, mycolor2}]
  table[row sep=crcr]{%
-4.50514997831991	-5.5926484623658\\
-4.80617997398389	-5.89368271305475\\
-5.10720996964787	-6.19471483724242\\
-5.40823996531185	-6.49574589259063\\
-5.70926996097583	-6.79677640693406\\
-6.01029995663981	-7.09780666176814\\
-6.31132995230379	-7.39883681008059\\
-6.61235994796777	-7.69986675018113\\
};
\addlegendentry{explicit Euler}

\addplot [color=black, draw=none, mark size=4.0pt, mark=triangle, mark options={solid, black}]
  table[row sep=crcr]{%
-4.50514997831991	-6.48742277468779\\
-4.80617997398389	-6.78845062852403\\
-5.10720996964787	-7.08947577519423\\
-5.40823996531185	-7.39049427047066\\
-5.70926996097583	-7.69163189888726\\
-6.01029995663981	-7.99264816466707\\
-6.31132995230379	-8.29334401486249\\
-6.61235994796777	-8.59456320713261\\
};
\addlegendentry{Lie splitting}

\end{axis}
\end{tikzpicture}%
 } & 
\resizebox{0.48\textwidth}{!}{
%
%
\definecolor{mycolor1}{rgb}{0.00000,0.44700,0.74100}%
\definecolor{mycolor2}{rgb}{0.2,0.6,0.4}%

\begin{tikzpicture}[scale=0.5]

\begin{axis}[%
width=10cm,
height=7.5cm,
at={(0.772in,0.473in)},
scale only axis,
xmin=-9,
xmax=-5.5,
xlabel style={font=\color{white!15!black}},
xlabel={log10 of the global error},
ymin=0,
ymax=20,
ylabel style={font=\color{white!15!black}},
ylabel={CPU time},
axis background/.style={fill=white},
title style={font=\bfseries},
legend style={legend cell align=left, align=left, draw=white!15!black}
]
\addplot [color=mycolor1, mark size=5pt, mark=diamond, mark options={solid, mycolor1}]
  table[row sep=crcr]{%
-5.96636311107174	0.100000000000009\\
-6.26739428806687	0.170000000000002\\
-6.56842486934128	0.299999999999997\\
-6.86945515813429	0.799999999999997\\
-7.17048527352206	1.5\\
-7.4715153386196	2.84999999999999\\
-7.77254549028603	4.74000000000001\\
-8.07357554092571	10.71\\
};
\addlegendentry{Linearly implicit}

\addplot [color=red, mark size=5.0pt, mark=o, mark options={solid, red}]
  table[row sep=crcr]{%
-5.59266548544308	0.200000000000003\\
-5.89369122568333	0.310000000000002\\
-6.19471909047088	0.569999999999993\\
-6.49574803382238	1.22\\
-6.79677749193149	2.27000000000001\\
-7.09780720185103	4.88\\
-7.39883712202114	9.78\\
-7.69986712885821	19.16\\
};
\addlegendentry{implicit Euler}

\addplot [color= mycolor2, mark size=5.0pt, mark=asterisk, mark options={solid,  mycolor2}]
  table[row sep=crcr]{%
-5.5926484623658	0.0900000000000034\\
-5.89368271305475	0.189999999999998\\
-6.19471483724242	0.310000000000002\\
-6.49574589259063	0.629999999999995\\
-6.79677640693406	1.27000000000001\\
-7.09780666176814	2.84\\
-7.39883681008059	5.13\\
-7.69986675018113	11.92\\
};
\addlegendentry{explicit Euler}

\addplot [color=black, mark size=4.0pt, mark=triangle, mark options={solid, black}]
  table[row sep=crcr]{%
-6.48742277468779	0.0699999999999932\\
-6.78845062852403	0.109999999999985\\
-7.08947577519423	0.29000000000002\\
-7.39049427047066	0.429999999999978\\
-7.69163189888726	0.860000000000014\\
-7.99264816466707	1.59\\
-8.29334401486249	3.00999999999999\\
-8.59456320713261	5.62\\
};
\addlegendentry{Lie splitting}

\end{axis}
\end{tikzpicture}%
}
  \end{tabular}
  \caption{Comparison of methods of order 1 applied to \eqref{eq:EDO3}: On the left hand side, maximal numerical error as a function of the time step (logarithmic scales); on the right hand side, CPU time (in seconds) as a function of the maximal numerical error. }
  \label{fig:ordre1_EDO}
\end{figure}

We then consider methods of order 2. We compare the linearly implicit method of order 2 defined in Section \ref{subsec:examples} for Gauss points with other methods of the literature: 
the midpoint method with Butcher tableau
\[
\renewcommand\arraystretch{1.2}
\begin{array}
{c|c}
1/2 & 1/2\\
\hline
& 1 
\end{array},
\]
the RK2 method with  Butcher tableau
\[
\renewcommand\arraystretch{1.2}
\begin{array}
{c|cc}
0 & 0 & 0\\
1/2 & 1/2 & 0\\
\hline
& 0 & 1 
\end{array},
\]
and the Strang splitting method. 
We choose $u^0=u_0=0.9$, $\gamma_{-1+c_1}=N(u((-1+c_1)h)),~\gamma_{-1+c_2}=N(u((-1+c_2)h))$ and the final time $T=2$.

The results are displayed in Figure \ref{fig:ordre2_EDO}. Once again the four methods are of order 2. This is a consequence of Theorem \ref{th:conv} for the linearly implicit method. The CPU time required for a given numerical error is much lower for the Strang splitting scheme than for the other three methods which perform similarly.

\begin{figure}[!h]
  \centering
   \begin{tabular}{cc}
\resizebox{0.49\textwidth}{!}{
%
%
\definecolor{mycolor1}{rgb}{0.00000,0.44700,0.74100}%
\definecolor{mycolor2}{rgb}{0.2,0.6,0.4}%

\begin{tikzpicture}

\begin{axis}[%
width=10cm,
height=7.5cm,
at={(0.758in,0.481in)},
scale only axis,
xmin=-5,
xmax=-1.5,
xlabel style={font=\color{white!15!black}},
xlabel={log10 of the time step},
ymin=-13,
ymax=-5,
ylabel style={font=\color{white!15!black}},
ylabel={log10 of the global error},
axis background/.style={fill=white},
legend style={at={(0.03,0.97)}, anchor=north west, legend cell align=left, align=left, draw=white!15!black}
]
\addplot [color=mycolor1, draw=none, mark size=5pt, mark=diamond, mark options={solid, mycolor1}]
  table[row sep=crcr]{%
-1.90308998699194	-5.59102600732413\\
-2.20411998265592	-6.19412279261806\\
-2.50514997831991	-6.79673092546809\\
-2.80617997398389	-7.39906888359179\\
-3.10720996964787	-8.00126967862436\\
-3.40823996531185	-8.60340060537913\\
-3.70926996097583	-9.20549555665516\\
-4.01029995663981	-9.80757458657062\\
-4.31132995230379	-10.4096375962879\\
-4.61235994796777	-11.0117065655131\\
-4.91338994363176	-11.6129722671281\\
};
\addlegendentry{Linearly implicit}

\addplot [color=red, draw=none, mark size=5.0pt, mark=o, mark options={solid, red}]
  table[row sep=crcr]{%
-1.90308998699194	-5.56902767738796\\
-2.20411998265592	-6.17108279199145\\
-2.50514997831991	-6.77314547723394\\
-2.80617997398389	-7.37520614291003\\
-3.10720996964787	-7.97726629896441\\
-3.40823996531185	-8.57932631202711\\
-3.70926996097583	-9.18138613862865\\
-4.01029995663981	-9.78344752098174\\
-4.31132995230379	-10.3855273532058\\
-4.61235994796777	-10.9876189746999\\
-4.91338994363176	-11.5893358432319\\
};
\addlegendentry{Midpoint}

\addplot [color=mycolor2, draw=none, mark size=5.0pt, mark=asterisk, mark options={solid, mycolor2}]
  table[row sep=crcr]{%
-1.90308998699194	-5.17109419609133\\
-2.20411998265592	-5.77637154359796\\
-2.50514997831991	-6.38003763755386\\
-2.80617997398389	-6.98289907993005\\
-3.10720996964787	-7.58536000722824\\
-3.40823996531185	-8.18762040681424\\
-3.70926996097583	-8.78978059471086\\
-4.01029995663981	-9.39189010884149\\
-4.31132995230379	-9.99397476709849\\
-4.61235994796777	-10.5960225139444\\
-4.91338994363176	-11.198037334441\\
};
\addlegendentry{RK2}

\addplot [color=black, draw=none, mark size=4.0pt, mark=triangle, mark options={solid, black}]
  table[row sep=crcr]{%
-1.90308998699194	-6.86763897899263\\
-2.20411998265592	-7.46969774244254\\
-2.50514997831991	-8.07175255460909\\
-2.80617997398389	-8.67381276349685\\
-3.10720996964787	-9.27588503504335\\
-3.40823996531185	-9.87790346045594\\
-3.70926996097583	-10.4798600888902\\
-4.01029995663981	-11.0844405461764\\
-4.31132995230379	-11.6597984959747\\
-4.61235994796777	-12.3966917588404\\
-4.91338994363176	-12.4002823613854\\
};
\addlegendentry{Strang splitting}

\end{axis}
\end{tikzpicture}%
 } & 
\resizebox{0.48\textwidth}{!}{
%
%
\definecolor{mycolor1}{rgb}{0.00000,0.44700,0.74100}%
\definecolor{mycolor2}{rgb}{0.2,0.6,0.4}%

\begin{tikzpicture}

\begin{axis}[%
width=10cm,
height=7.5cm,
at={(0.745in,0.489in)},
scale only axis,
xmin=-13,
xmax=-5,
xlabel style={font=\color{white!15!black}},
xlabel={log10 of the global error},
ymin=0,
ymax=0.5,
ylabel style={font=\color{white!15!black}},
ylabel={CPU time},
axis background/.style={fill=white},
legend style={legend cell align=left, align=left, draw=white!15!black}
]
\addplot [color=mycolor1, mark size=5pt, mark=diamond, mark options={solid, mycolor1}]
  table[row sep=crcr]{%
-5.59102600732413	0.200000000000003\\
-6.19412279261806	0\\
-6.79673092546809	0.0100000000000051\\
-7.39906888359179	0.0300000000000011\\
-8.00126967862436	0.0399999999999991\\
-8.60340060537913	0.0499999999999972\\
-9.20549555665516	0.0600000000000023\\
-9.80757458657062	0.109999999999999\\
-10.4096375962879	0.140000000000001\\
-11.0117065655131	0.18\\
-11.6129722671281	0.420000000000002\\
};
\addlegendentry{Linearly implicit}

\addplot [color=red, mark size=5.0pt, mark=o, mark options={solid, red}]
  table[row sep=crcr]{%
-5.56902767738796	0.019999999999996\\
-6.17108279199145	0\\
-6.77314547723394	0.0200000000000031\\
-7.37520614291003	0.00999999999999801\\
-7.97726629896441	0.0399999999999991\\
-8.57932631202711	0.0200000000000031\\
-9.18138613862865	0.0499999999999972\\
-9.78344752098174	0.0799999999999983\\
-10.3855273532058	0.120000000000005\\
-10.9876189746999	0.229999999999997\\
-11.5893358432319	0.460000000000001\\
};
\addlegendentry{Midpoint}

\addplot [color= mycolor2, mark size=5.0pt, mark=asterisk, mark options={solid,  mycolor2}]
  table[row sep=crcr]{%
-5.17109419609133	0.00999999999999801\\
-5.77637154359796	0.0100000000000051\\
-6.38003763755386	0.00999999999999801\\
-6.98289907993005	0.00999999999999801\\
-7.58536000722824	0.00999999999999801\\
-8.18762040681424	0.0200000000000031\\
-8.78978059471086	0.0499999999999972\\
-9.39189010884149	0.0399999999999991\\
-9.99397476709849	0.0800000000000054\\
-10.5960225139444	0.149999999999999\\
-11.198037334441	0.269999999999996\\
};
\addlegendentry{RK2}

\addplot [color=black, mark size=4.0pt, mark=triangle, mark options={solid, black}]
  table[row sep=crcr]{%
-6.86763897899263	0.0100000000000051\\
-7.46969774244254	0\\
-8.07175255460909	0\\
-8.67381276349685	0\\
-9.27588503504335	0.00999999999999801\\
-9.87790346045594	0.0200000000000031\\
-10.4798600888902	0.0300000000000011\\
-11.0844405461764	0.0499999999999972\\
-11.6597984959747	0.0399999999999991\\
-12.3966917588404	0.0600000000000023\\
-12.4002823613854	0.140000000000001\\
};
\addlegendentry{Strang splitting}

\end{axis}
\end{tikzpicture}%
}
  \end{tabular}
  \caption{Comparison of methods of order 2 applied to \eqref{eq:EDO3}: On the left hand side, maximal numerical error as a function of the time step (logarithmic scales); on the right hand side, CPU time (in seconds) as a function of the maximal numerical error. }
  \label{fig:ordre2_EDO}
\end{figure}
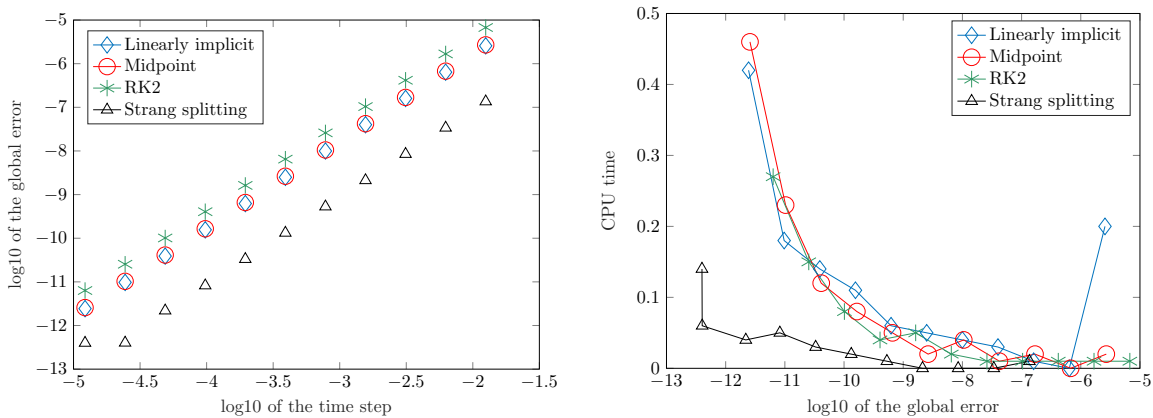

Similar results are obtained (but not displayed here) for methods of order four and six which illustrate Theorem \ref{th:conv} for the corresponding linearly implicit methods introduced in Section \ref{subsec:examples}. Moreover the CPU time required to reach a given numerical error is always higher for the linearly implicit schemes than for other classical methods of the same order for the ODE \eqref{eq:EDO3}.

\eject
\subsection{Application to the nonlinear Schr\"odinger equation}
\label{subsec:Schro}

\subsubsection{One dimensional nonlinear Schr\"odinger equation: The soliton case}
\label{subsubsec:Schro1d}

In this section, we consider the nonlinear one dimensional Schr\"odinger equation:
\begin{equation}
\label{eq:NLS1D}
i\partial_t u = -\partial^2_x u - q|u|^2 u,
\end{equation}
which corresponds to the evolution problem \eqref{eq:PDE} with $L=i\partial^2_x$ and $N(u)=iq|u|^2$. We consider the initial condition 
\begin{equation*}
u_0(x)=\sqrt{\dfrac{2a}{q}}{\rm sech}\left(\sqrt{a}x\right),
\end{equation*}
where $q>0$ and $a=q^2/16$, so that the corresponding exact solution of \eqref{eq:NLS1D} is the zero speed soliton and reads
\begin{equation}
\label{eq:soliton1D}
u(t,x)=\sqrt{\dfrac{2a}{q}}{\rm sech}\left(\sqrt{a}x\right)\exp(iat).
\end{equation}
We use $q=4$ and $a=1$ for the numerical simulations. The final time is set to $T=5$.
For the space discretization, we consider the interval $[-50,50]$ with homogeneous
Dirichlet boundary conditions
since the exact solution \eqref{eq:soliton1D} decays very fast when $|x|$ tends to $+\infty$.
We use $2^{14}$ equispaced points in space for methods of order 1 in time and $2^{18}$ equispaced points in space for methods of order 2 in time. For splitting methods, one has to integrate numerically equation \eqref{eq:NLS1D} with $q=0$. This is done {\it via} the approximation 
\begin{equation}
\label{eq:approxexp}
\exp\left(ihB\right)=\left(I+i\dfrac{h}{2} B\right)\left(I-i\dfrac{h}{2} B\right)^{-1}+\mathcal{O}(h^3),
\end{equation}
where $I$ denotes the identity matrix, $B$ the Laplacian with Dirichlet boundary conditions matrix,
and $h>0$ the small time step.
In particular, the splitting methods we use below are also linearly implicit
(the nonlinear part of the equation is integrated exactly).
The numerical error we consider for all numerical methods is the discrete $L^2$-norm of the difference between the numerical solution and the projection of the exact solution \eqref{eq:soliton1D} on the space grid at final time $T$. 

First we compare methods of order one. We consider the linearly implicit method of order 1 introduced in Section \ref{subsec:examples}, the implicit Euler scheme and the Lie splitting scheme. For the linearly implicit method, we initialize the scheme with $u^0=u_0=u(0,\cdot)$, $\gamma_{-1+c_1}=N(u((-1+c_1)h,\cdot))$. The results are given in Figure \ref{fig:ordre1_NLS1Dsoliton}. The figure on the left hand side shows that the three methods are of order 1. This illustrates the fact that the conclusion of Theorem \ref{th:conv} for the linearly implicit method holds numerically in this PDE context. For such a simulation, we can see, on the figure on the right hand side, that now the CPU time required to reach a given error is smaller for the linearly implicit method than for the fully implicit Euler method. However the Lie splitting method is the least time consuming method since it is explicit (in fact our implementation of the Lie splitting method makes it linearly implicit, see \eqref{eq:approxexp})
and has a good error constant. 

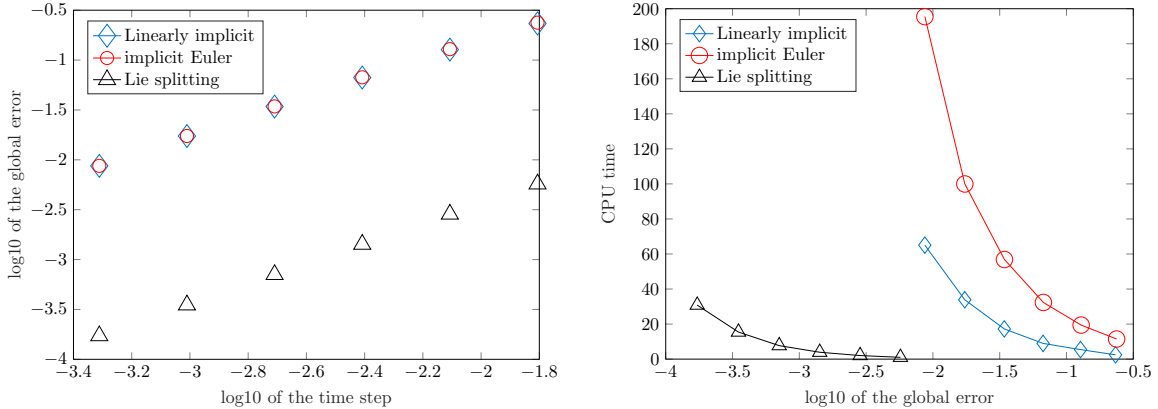
\begin{figure}[!h]
  \centering
   \begin{tabular}{cc}
\resizebox{0.49\textwidth}{!}{
%
%
\definecolor{mycolor1}{rgb}{0.00000,0.44700,0.74100}%
\begin{tikzpicture}

\begin{axis}[%
width=10cm,
height=7.5cm,
at={(1.011in,0.642in)},
scale only axis,
xmin=-3.4,
xmax=-1.8,
xlabel style={font=\color{white!15!black}},
xlabel={log10 of the time step},
ymin=-4,
ymax=-0.5,
ylabel style={font=\color{white!15!black}},
ylabel={log10 of the global error},
axis background/.style={fill=white},
legend style={at={(0.03,0.97)}, anchor=north west, legend cell align=left, align=left, draw=white!15!black}
]
\addplot [color=mycolor1, draw=none, mark size=7.0pt, mark=diamond, mark options={mycolor1}]
  table[row sep=crcr]{%
-1.80617997398389	-0.632851564436005\\
-2.10720996964787	-0.894941706449429\\
-2.40823996531185	-1.17528311865359\\
-2.70926996097583	-1.46573203982411\\
-3.01029995663981	-1.76161963710205\\
-3.31132995230379	-2.06054121633667\\
};
\addlegendentry{Linearly implicit}

\addplot [color=red, draw=none, mark size=4.0pt, mark=o, mark options={red}]
  table[row sep=crcr]{%
-1.80617997398389	-0.624719314656724\\
-2.10720996964787	-0.890480844371899\\
-2.40823996531185	-1.17293906987317\\
-2.70926996097583	-1.46452916549013\\
-3.01029995663981	-1.76100995984052\\
-3.31132995230379	-2.0602340594404\\
};
\addlegendentry{implicit Euler}

\addplot [color=black, draw=none, mark size=6.0pt, mark=triangle, mark options={black}]
  table[row sep=crcr]{%
-1.80617997398389	-2.24281019405233\\
-2.10720996964787	-2.54496884036537\\
-2.40823996531185	-2.84728933372178\\
-2.70926996097583	-3.15043670338034\\
-3.01029995663981	-3.45545124495294\\
-3.31132995230379	-3.76420685658136\\
};
\addlegendentry{Lie splitting}

\end{axis}
\end{tikzpicture}%
 } & 
\resizebox{0.48\textwidth}{!}{
%
%
\definecolor{mycolor1}{rgb}{0.00000,0.44700,0.74100}%
\begin{tikzpicture}

\begin{axis}[%
width=10cm,
height=7.5cm,
at={(1.011in,0.642in)},
scale only axis,
xmin=-4,
xmax=-0.5,
xlabel style={font=\color{white!15!black}},
xlabel={log10 of the global error},
ymin=0,
ymax=200,
ylabel style={font=\color{white!15!black}},
ylabel={CPU time},
axis background/.style={fill=white},
legend style={at={(0.03,0.97)}, anchor=north west, legend cell align=left, align=left, draw=white!15!black}
]
\addplot [color=mycolor1, mark size=5.0pt, mark=diamond, mark options={solid, mycolor1}]
  table[row sep=crcr]{%
-0.632851564436005	2.44999999999999\\
-0.894941706449429	5.43000000000001\\
-1.17528311865359	8.97\\
-1.46573203982411	17.23\\
-1.76161963710205	33.88\\
-2.06054121633667	65.06\\
};
\addlegendentry{Linearly implicit}

\addplot [color=red, mark size=5.0pt, mark=o, mark options={solid, red}]
  table[row sep=crcr]{%
-0.624719314656724	11.47\\
-0.890480844371899	19.43\\
-1.17293906987317	32.34\\
-1.46452916549013	56.85\\
-1.76100995984052	99.97\\
-2.0602340594404	195.48\\
};
\addlegendentry{implicit Euler}

\addplot [color=black, mark size=5.0pt, mark=triangle, mark options={solid, black}]
  table[row sep=crcr]{%
-2.24281019405233	1.03000000000009\\
-2.54496884036537	2.02999999999997\\
-2.84728933372178	3.91999999999996\\
-3.15043670338034	7.69000000000005\\
-3.45545124495294	15.4399999999999\\
-3.76420685658136	30.8200000000001\\
};
\addlegendentry{Lie splitting}

\end{axis}
\end{tikzpicture}%
}
  \end{tabular}
  \caption{Comparison of methods of order 1 applied to \eqref{eq:NLS1D}: On the left hand side, maximal numerical error as a function of the time step (logarithmic scales); on the right hand side, CPU time (in seconds) as a function of the maximal numerical error. }
  \label{fig:ordre1_NLS1Dsoliton}
\end{figure}


We then consider methods of order 2. For this experiment we use the linearly implicit method of order 2 introduced in Section \ref{subsec:examples} for the uniform points, the Crank-Nicolson scheme \cite{CN47,DFP81} (for which we solve the nonlinear system using a fixed point algorithm) and the Strang splitting method \cite{Lubich08}. For the implementation of the Strang splitting method, we use the classical conjugation with the Lie splitting method which we recall briefly below and relies on the identity
\begin{equation}
\label{eq:LieStrang}
\left(\exp\left(i\dfrac{h}{2}B\right)\circ \Phi_h \circ \exp\left(i\dfrac{h}{2}B\right)\right)^k=\exp\left(i\dfrac{h}{2}B\right) \circ \left(\Phi_h \circ \exp\left(ihB\right)\right)^k\circ \exp\left(-i\dfrac{h}{2}B\right),
\end{equation}
where $\Phi_h$ denotes the numerical flow of the nonlinear part of \eqref{eq:NLS1D} defined componentwise using the function $v \mapsto \exp(ihq|v|^2)v$, and $k$ is any nonnegative integer. The numerical flow of the Lie splitting method is $\Phi_h^{Lie}=\Phi_h \circ \exp(ihB)$ and that of the Strang splitting method is $\Phi_h^{Strang}=\exp\left(ihB/2\right)\circ \Phi_h \circ \exp\left(ihB/2\right)$. Therefore, relation \eqref{eq:LieStrang} also reads
\begin{equation}
\label{eq:LieStrang2}
\left(\Phi_h^{Strang}\right)^k=\exp\left(i\dfrac{h}{2}B\right) \circ \left(\Phi_h^{Lie}\right)^k\circ \exp\left(-i\dfrac{h}{2}B\right),
\end{equation}
for all nonnegative integer $k$. The implementation of the Strang splitting method we use for numerical simulations uses both the right hand side of the equation \eqref{eq:LieStrang2} and the approximation formula \eqref{eq:approxexp}.
For the linearly implicit method, we initialize the scheme with $u^0=u_0=u(0,\cdot)$, $\gamma_{-1+c_1}=N(u((-1+c_1)h,\cdot)),~\gamma_{-1+c_2}=N(u((-1+c_2)h,\cdot))$.
The results are displayed in Figure \ref{fig:ordre2_NLS1Dsoliton}. The numerical order of each method is the one expected {\it i.e.} 2. This illustrates once again the numerical relevance of Theorem \ref{th:conv} beyond the ODE context. As we can see on the figure on the right hand side, the CPU time required to reach a given error for the Crank-Nicolson method is higher than the one for  the linearly implicit method of order 2 which is also a little higher than the one for the Strang splitting method. 

\begin{remark}
  For this example, if instead of using the linearly implicit method of order 2 with uniform points, we use the linearly implicit method of order 2 with Gauss points, we numerically obtain the superconvergence of the method and a numerical order equal to 4. This is due to the fact that in this particular case the modulus of the solution is constant in time as it can be seen on \eqref{eq:soliton1D}. 
\end{remark}

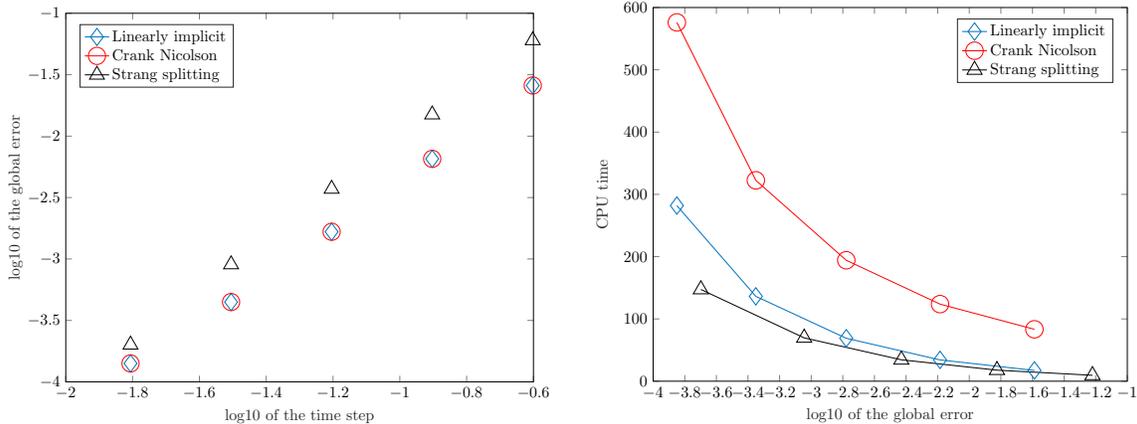
\begin{figure}[!h]
  \centering
   \begin{tabular}{cc}
\resizebox{0.48\textwidth}{!}{
%
%
\definecolor{mycolor1}{rgb}{0.00000,0.44700,0.74100}%
\begin{tikzpicture}

\begin{axis}[%
width=4.521in,
height=3.566in,
at={(1.011in,0.642in)},
scale only axis,
xmin=-2,
xmax=-0.6,
xlabel style={font=\color{white!15!black}},
xlabel={log10 of the time step},
ymin=-4,
ymax=-1,
ylabel style={font=\color{white!15!black}},
ylabel={log10 of the global error},
axis background/.style={fill=white},
legend style={at={(0.03,0.97)}, anchor=north west, legend cell align=left, align=left, draw=white!15!black}
]
\addplot [color=mycolor1, draw=none, mark size=6.0pt, mark=diamond, mark options={solid, mycolor1}]
  table[row sep=crcr]{%
-0.602059991327962	-1.5877560093991\\
-0.903089986991944	-2.18481135182164\\
-1.20411998265592	-2.77833043005767\\
-1.50514997831991	-3.35039240695028\\
-1.80617997398389	-3.85007788412759\\
};
\addlegendentry{Linearly implicit}

\addplot [color=red, draw=none, mark size=6.0pt, mark=o, mark options={solid, red}]
  table[row sep=crcr]{%
-0.602059991327962	-1.58773584053535\\
-0.903089986991944	-2.18478900561259\\
-1.20411998265592	-2.77830762668812\\
-1.50514997831991	-3.3503708007979\\
-1.80617997398389	-3.85006067204541\\
};
\addlegendentry{Crank Nicolson}

\addplot [color=black, draw=none, mark size=6.0pt, mark=triangle, mark options={solid, black}]
  table[row sep=crcr]{%
-0.602059991327962	-1.22052314995563\\
-0.903089986991944	-1.824757707846\\
-1.20411998265592	-2.42969715316417\\
-1.50514997831991	-3.04469888612012\\
-1.80617997398389	-3.69811041763705\\
};
\addlegendentry{Strang splitting}

\end{axis}
\end{tikzpicture}%
} 
& 
\resizebox{0.48\textwidth}{!}{
%
%
\definecolor{mycolor1}{rgb}{0.00000,0.44700,0.74100}%
\begin{tikzpicture}

\begin{axis}[%
width=4.521in,
height=3.566in,
at={(1.011in,0.642in)},
scale only axis,
xmin=-4,
xmax=-1,
xlabel style={font=\color{white!15!black}},
xlabel={log10 of the global error},
ymin=0,
ymax=600,
ylabel style={font=\color{white!15!black}},
ylabel={CPU time},
axis background/.style={fill=white},
legend style={at={(0.97,0.97)}, anchor=north east, legend cell align=left, align=left, draw=white!15!black}
]
\addplot [color=mycolor1, mark size=6.0pt, mark=diamond, mark options={solid, mycolor1}]
  table[row sep=crcr]{%
-1.5877560093991	17.3300000000008\\
-2.18481135182164	34.2199999999993\\
-2.77833043005767	68.9800000000005\\
-3.35039240695028	136.28\\
-3.85007788412759	282.01\\
};
\addlegendentry{Linearly implicit}

\addplot [color=red, mark size=6.0pt, mark=o, mark options={solid, red}]
  table[row sep=crcr]{%
-1.58773584053535	83.3499999999995\\
-2.18478900561259	123.780000000001\\
-2.77830762668812	194.19\\
-3.3503708007979	322.51\\
-3.85006067204541	576.07\\
};
\addlegendentry{Crank Nicolson}

\addplot [color=black, mark size=6.0pt, mark=triangle, mark options={solid, black}]
  table[row sep=crcr]{%
-1.22052314995563	9.5\\
-1.824757707846	17.75\\
-2.42969715316417	34.4000000000005\\
-3.04469888612012	69.6099999999997\\
-3.69811041763705	147.31\\
};
\addlegendentry{Strang splitting}

\end{axis}
\end{tikzpicture}%
}
\end{tabular}
  \caption{Comparison of methods of order 2 applied to \eqref{eq:NLS1D}: On the left hand side, maximal numerical error as a function of the time step (logarithmic scales); on the right hand side, CPU time (in seconds) as a function of the maximal numerical error. }
  \label{fig:ordre2_NLS1Dsoliton}
\end{figure}

\begin{figure}[!h]
  \centering
   \begin{tabular}{cc}
\resizebox{0.48\textwidth}{!}{
%
%
\definecolor{mycolor1}{rgb}{0.00000,0.44700,0.74100}%

\begin{tikzpicture}

\begin{axis}[%
width=6.028in,
height=4.754in,
at={(1.011in,0.642in)},
scale only axis,
xmin=0,
xmax=5,
xlabel style={font=\color{white!15!black}},
xlabel style={font=\Large},
xlabel={time},
every x tick label/.append style={font=\Large},
ymin=-1.2e-09,
ymax=2e-10,
ylabel style={font=\color{white!15!black}},
ylabel style={font=\Large},
ylabel={squared $L^2$-norm variation},
every y tick label/.append style={font=\Large},
axis background/.style={fill=white},
legend style={at={(0.634,0.621)}, anchor=south west, legend cell align=left, align=left, draw=white!15!black, font=\Large}
]
\addplot [color=mycolor1, draw=none, mark size=6pt, mark=diamond, mark options={solid, mycolor1}]
  table[row sep=crcr]{%
0.015625	1.11022302462516e-15\\
0.09375	-1.33152378012369e-11\\
0.171875	-5.28308508052078e-11\\
0.25	-1.09218079025197e-10\\
0.328125	-1.73405401149296e-10\\
0.40625	-2.39794628598133e-10\\
0.484375	-3.05411473888739e-10\\
0.5625	-3.68797881122873e-10\\
0.640625	-4.29300595072846e-10\\
0.71875	-4.86670481691931e-10\\
0.796875	-5.40854250274947e-10\\
0.875	-5.91882987066583e-10\\
0.953125	-6.39823638515225e-10\\
1.03125	-6.84750922630428e-10\\
1.109375	-7.26735227551956e-10\\
1.1875	-7.65835506122414e-10\\
1.265625	-8.02108712782967e-10\\
1.34375	-8.35596369874736e-10\\
1.421875	-8.66338445426607e-10\\
1.5	-8.94367579995503e-10\\
1.578125	-9.19720188896633e-10\\
1.65625	-9.4242480486173e-10\\
1.734375	-9.62513069246995e-10\\
1.8125	-9.80019621010797e-10\\
1.890625	-9.94986204538861e-10\\
1.96875	-1.00745001230251e-09\\
2.046875	-1.01745889491411e-09\\
2.125	-1.02506692023496e-09\\
2.203125	-1.0303301545278e-09\\
2.28125	-1.03331299072806e-09\\
2.359375	-1.03408659413162e-09\\
2.4375	-1.0327242394581e-09\\
2.515625	-1.02931063672429e-09\\
2.59375	-1.02393460377215e-09\\
2.671875	-1.01668828911272e-09\\
2.75	-1.00767261201895e-09\\
2.828125	-9.96992044477452e-10\\
2.90625	-9.84754278121613e-10\\
2.984375	-9.71077218636651e-10\\
3.0625	-9.56078660685478e-10\\
3.140625	-9.39878286310147e-10\\
3.21875	-9.22601550712443e-10\\
3.296875	-9.04376129540196e-10\\
3.375	-8.85329365374332e-10\\
3.453125	-8.65593374754781e-10\\
3.53125	-8.45296277418583e-10\\
3.609375	-8.24571522173301e-10\\
3.6875	-8.03548450001301e-10\\
3.765625	-7.82355846773441e-10\\
3.84375	-7.61121277115251e-10\\
3.921875	-7.39970418273117e-10\\
4	-7.19026616025076e-10\\
4.078125	-6.98406443788713e-10\\
4.15625	-6.78225253736287e-10\\
4.234375	-6.58592846924932e-10\\
4.3125	-6.39617026010342e-10\\
4.390625	-6.21393714261842e-10\\
4.46875	-6.04017502681131e-10\\
4.546875	-5.87578652400111e-10\\
4.625	-5.72155545164321e-10\\
4.703125	-5.57821677737991e-10\\
4.78125	-5.44642664301875e-10\\
4.859375	-5.32680011211539e-10\\
4.9375	-5.21983900547696e-10\\
};
\addlegendentry{Linearly implicit}

\addplot [color=red, draw=none, mark size=6.0pt, mark=o, mark options={solid, red}]
  table[row sep=crcr]{%
0.015625	1.11022302462516e-15\\
0.09375	8.88178419700125e-16\\
0.171875	-1.22124532708767e-15\\
0.25	8.88178419700125e-16\\
0.328125	-5.55111512312578e-16\\
0.40625	-1.22124532708767e-15\\
0.484375	8.88178419700125e-16\\
0.5625	-1.22124532708767e-15\\
0.640625	1.99840144432528e-15\\
0.71875	2.66453525910038e-15\\
0.796875	-8.88178419700125e-16\\
0.875	-4.21884749357559e-15\\
0.953125	2.22044604925031e-16\\
1.03125	-1.11022302462516e-15\\
1.109375	-6.21724893790088e-15\\
1.1875	-6.66133814775094e-15\\
1.265625	-8.32667268468867e-15\\
1.34375	-1.23234755733392e-14\\
1.421875	-1.35447209004269e-14\\
1.5	-1.47659662275146e-14\\
1.578125	-1.57651669496772e-14\\
1.65625	-1.59872115546023e-14\\
1.734375	-1.66533453693773e-14\\
1.8125	-1.69864122767649e-14\\
1.890625	-1.28785870856518e-14\\
1.96875	-9.88098491916389e-15\\
2.046875	-1.50990331349021e-14\\
2.125	-2.08721928629529e-14\\
2.203125	-2.57571741713036e-14\\
2.28125	-2.27595720048157e-14\\
2.359375	-2.08721928629529e-14\\
2.4375	-1.80966353013901e-14\\
2.515625	-1.4432899320127e-14\\
2.59375	-9.54791801177635e-15\\
2.671875	-6.99440505513849e-15\\
2.75	-8.88178419700125e-15\\
2.828125	-1.02140518265514e-14\\
2.90625	-1.13242748511766e-14\\
2.984375	-1.0991207943789e-14\\
3.0625	-1.13242748511766e-14\\
3.140625	-1.11022302462516e-14\\
3.21875	-9.76996261670138e-15\\
3.296875	-1.01030295240889e-14\\
3.375	-9.54791801177635e-15\\
3.453125	-7.88258347483861e-15\\
3.53125	-4.66293670342566e-15\\
3.609375	1.77635683940025e-15\\
3.6875	-2.22044604925031e-16\\
3.765625	-7.43849426498855e-15\\
3.84375	-3.33066907387547e-15\\
3.921875	-3.33066907387547e-15\\
4	-4.32986979603811e-15\\
4.078125	-2.77555756156289e-15\\
4.15625	1.11022302462516e-15\\
4.234375	-8.88178419700125e-16\\
4.3125	2.66453525910038e-15\\
4.390625	3.99680288865056e-15\\
4.46875	4.21884749357559e-15\\
4.546875	3.10862446895044e-15\\
4.625	2.88657986402541e-15\\
4.703125	2.66453525910038e-15\\
4.78125	3.77475828372553e-15\\
4.859375	1.99840144432528e-15\\
4.9375	4.44089209850063e-15\\
};
\addlegendentry{Crank-Nicolson}

\addplot [color=black, draw=none, mark size=6.0pt, mark=triangle, mark options={solid, black}]
  table[row sep=crcr]{%
0.015625	2.22044604925031e-16\\
0.09375	4.66293670342566e-15\\
0.171875	1.06581410364015e-14\\
0.25	1.68753899743024e-14\\
0.328125	2.46469511466785e-14\\
0.40625	3.41948691584548e-14\\
0.484375	4.10782519111308e-14\\
0.5625	5.08482145278322e-14\\
0.640625	5.97299987248334e-14\\
0.71875	6.3948846218409e-14\\
0.796875	6.63913368725844e-14\\
0.875	7.32747196252603e-14\\
0.953125	7.54951656745106e-14\\
1.03125	8.10462807976364e-14\\
1.109375	8.72635297355373e-14\\
1.1875	9.48130463029884e-14\\
1.265625	1.00808250635964e-13\\
1.34375	1.06581410364015e-13\\
1.421875	1.13020703906841e-13\\
1.5	1.19015908239817e-13\\
1.578125	1.24789067967868e-13\\
1.65625	1.30118138486068e-13\\
1.734375	1.38111744263369e-13\\
1.8125	1.42330591756945e-13\\
1.890625	1.46993528460371e-13\\
1.96875	1.56541446472147e-13\\
2.046875	1.61204383175573e-13\\
2.125	1.72750702631674e-13\\
2.203125	1.79412040779425e-13\\
2.28125	1.83186799063151e-13\\
2.359375	1.92512672470002e-13\\
2.4375	1.95399252334028e-13\\
2.515625	1.98063787593128e-13\\
2.59375	2.06945571790129e-13\\
2.671875	2.15161222172355e-13\\
2.75	2.17159623616681e-13\\
2.828125	2.21378471110256e-13\\
2.90625	2.31814567541733e-13\\
2.984375	2.40474307133809e-13\\
3.0625	2.4535928844216e-13\\
3.140625	2.50910403565285e-13\\
3.21875	2.57571741713036e-13\\
3.296875	2.63566946046012e-13\\
3.375	2.69562150378988e-13\\
3.453125	2.8022029141539e-13\\
3.53125	2.8554936193359e-13\\
3.609375	2.91988655476416e-13\\
3.6875	2.88213897192691e-13\\
3.765625	2.93542967710891e-13\\
3.84375	2.98427949019242e-13\\
3.921875	2.95319324550292e-13\\
4	2.93987056920741e-13\\
4.078125	2.98427949019242e-13\\
4.15625	2.95319324550292e-13\\
4.234375	2.93987056920741e-13\\
4.3125	2.95319324550292e-13\\
4.390625	2.93765012315816e-13\\
4.46875	2.95319324550292e-13\\
4.546875	2.93542967710891e-13\\
4.625	2.94431146130592e-13\\
4.703125	2.95319324550292e-13\\
4.78125	2.95763413760142e-13\\
4.859375	2.96207502969992e-13\\
4.9375	2.97983859809392e-13\\
};
\addlegendentry{Strang splitting}

\end{axis}

\end{tikzpicture}%
} 
& 
\resizebox{0.48\textwidth}{!}{
%
%
\definecolor{mycolor1}{rgb}{0.00000,0.44700,0.74100}%

\begin{tikzpicture}

\begin{axis}[%
width=6.028in,
height=4.754in,
at={(1.011in,0.642in)},
scale only axis,
xmin=0,
xmax=5,
xlabel style={font=\color{white!15!black}},
xlabel style={font=\Large},
xlabel={time},
every x tick label/.append style={font=\Large},
ymin=-1e-08,
ymax=7e-08,
ylabel style={font=\color{white!15!black}},
ylabel style={font=\Large},
ylabel={energy variation},
every y tick label/.append style={font=\Large},
axis background/.style={fill=white},
legend style={at={(0.659,0.536)}, anchor=south west, legend cell align=left, align=left, draw=white!15!black, font=\Large}
]
\addplot [color=mycolor1, draw=none, mark size=6.0pt, mark=diamond, mark options={solid, blue}]
  table[row sep=crcr]{%
0.015625	1.11022302462516e-15\\
0.09375	6.64451826892787e-12\\
0.171875	2.64021027263084e-11\\
0.25	5.45974654375669e-11\\
0.328125	8.66878235861179e-11\\
0.40625	1.19892706873514e-10\\
0.484375	1.52696161270782e-10\\
0.5625	1.84393361690738e-10\\
0.640625	2.14643303131368e-10\\
0.71875	2.43329356663935e-10\\
0.796875	2.70422656489799e-10\\
0.875	2.95937052641193e-10\\
0.953125	3.1990810001048e-10\\
1.03125	3.42371409001174e-10\\
1.109375	3.63363561461938e-10\\
1.1875	3.82917614283329e-10\\
1.265625	4.01054134346879e-10\\
1.34375	4.17794132623328e-10\\
1.421875	4.33162339330551e-10\\
1.5	4.47184927976352e-10\\
1.578125	4.59854515577618e-10\\
1.65625	4.71204852914298e-10\\
1.734375	4.81248652040023e-10\\
1.8125	4.90011642373389e-10\\
1.890625	4.97494018203426e-10\\
1.96875	5.03715819055728e-10\\
2.046875	5.08720399139406e-10\\
2.125	5.12532932761545e-10\\
2.203125	5.15161469039072e-10\\
2.28125	5.16653747562046e-10\\
2.359375	5.1704121539764e-10\\
2.4375	5.16361092772755e-10\\
2.515625	5.14654069361242e-10\\
2.59375	5.11963305083185e-10\\
2.671875	5.08349834449362e-10\\
2.75	5.03836639076383e-10\\
2.828125	4.98496410816784e-10\\
2.90625	4.92375418215119e-10\\
2.984375	4.85540496697467e-10\\
3.0625	4.78041495277637e-10\\
3.140625	4.69940864000762e-10\\
3.21875	4.61300248000285e-10\\
3.296875	4.52183901433756e-10\\
3.375	4.42667208444547e-10\\
3.453125	4.32789609705608e-10\\
3.53125	4.22638229968797e-10\\
3.609375	4.12283346351572e-10\\
3.6875	4.01770755553699e-10\\
3.765625	3.91175147829159e-10\\
3.84375	3.80557196866249e-10\\
3.921875	3.69983321757417e-10\\
4	3.59510893277459e-10\\
4.078125	3.49202111671332e-10\\
4.15625	3.39107214530898e-10\\
4.234375	3.29286098388337e-10\\
4.3125	3.19805210091673e-10\\
4.390625	3.10696551819589e-10\\
4.46875	3.02005392915916e-10\\
4.546875	2.93785579197348e-10\\
4.625	2.86070556132501e-10\\
4.703125	2.78906425732472e-10\\
4.78125	2.72317557392654e-10\\
4.859375	2.66336758203423e-10\\
4.9375	2.60988231026715e-10\\
};
\addlegendentry{Linearly implicit}

\addplot [color=red, draw=none, mark size=6.0pt, mark=o, mark options={solid, red}]
  table[row sep=crcr]{%
0.015625	-5.74540415243519e-15\\
0.09375	-5.6621374255883e-15\\
0.171875	-5.80091530366644e-15\\
0.25	4.94049245958195e-15\\
0.328125	-9.65894031423886e-15\\
0.40625	-4.71844785465692e-15\\
0.484375	-5.30131494258512e-15\\
0.5625	-3.33066907387547e-15\\
0.640625	-2.77555756156289e-15\\
0.71875	-3.77475828372553e-15\\
0.796875	-2.77555756156289e-15\\
0.875	-1.47104550762833e-15\\
0.953125	-3.02535774210355e-15\\
1.03125	-6.38378239159465e-16\\
1.109375	2.60902410786912e-15\\
1.1875	9.71445146547012e-16\\
1.265625	-2.22044604925031e-16\\
1.34375	1.01030295240889e-14\\
1.421875	2.35922392732846e-15\\
1.5	1.27675647831893e-15\\
1.578125	5.30131494258512e-15\\
1.65625	3.38618022510673e-15\\
1.734375	3.24740234702858e-15\\
1.8125	2.77555756156289e-16\\
1.890625	4.96824803519758e-15\\
1.96875	1.01585406753202e-14\\
2.046875	5.99520433297585e-15\\
2.125	8.46545056276682e-15\\
2.203125	7.85482789922298e-15\\
2.28125	1.03528297046296e-14\\
2.359375	6.49480469405717e-15\\
2.4375	7.04991620636974e-15\\
2.515625	5.6621374255883e-15\\
2.59375	-9.15933995315754e-16\\
2.671875	7.49400541621981e-16\\
2.75	7.52176099183544e-15\\
2.828125	-3.94129173741931e-15\\
2.90625	1.13797860024079e-15\\
2.984375	-3.60822483003176e-15\\
3.0625	-4.9960036108132e-16\\
3.140625	7.49400541621981e-16\\
3.21875	-2.3037127760972e-15\\
3.296875	1.04083408558608e-14\\
3.375	5.91193760612896e-15\\
3.453125	-2.66453525910038e-15\\
3.53125	3.02535774210355e-15\\
3.609375	-2.692290834716e-15\\
3.6875	-1.47104550762833e-15\\
3.765625	9.43689570931383e-16\\
3.84375	-2.99760216648792e-15\\
3.921875	8.32667268468867e-16\\
4	-1.11022302462516e-16\\
4.078125	-3.05311331771918e-16\\
4.15625	-3.49720252756924e-15\\
4.234375	5.35682609381638e-15\\
4.3125	-1.03805852802452e-14\\
4.390625	-6.43929354282591e-15\\
4.46875	-2.3037127760972e-15\\
4.546875	-6.38378239159465e-15\\
4.625	1.11022302462516e-16\\
4.703125	-3.63598040564739e-15\\
4.78125	-2.66453525910038e-15\\
4.859375	-3.05311331771918e-16\\
4.9375	-1.27675647831893e-15\\
};
\addlegendentry{Crank-Nicolson}

\addplot [color=black, draw=none, mark size=6.0pt, mark=triangle, mark options={solid, black}]
  table[row sep=crcr]{%
0.015625	1.16135331906797e-09\\
0.09375	2.9153858105424e-08\\
0.171875	5.30317449398687e-08\\
0.25	6.07513878370103e-08\\
0.328125	6.11527858906946e-08\\
0.40625	5.95594749075445e-08\\
0.484375	5.78353676894405e-08\\
0.5625	5.64565884597634e-08\\
0.640625	5.54586110923516e-08\\
0.71875	5.47670714334814e-08\\
0.796875	5.42988005391987e-08\\
0.875	5.39867635040991e-08\\
0.953125	5.37821474300149e-08\\
1.03125	5.36507696324851e-08\\
1.109375	5.35690508896014e-08\\
1.1875	5.35207857199627e-08\\
1.265625	5.34947802999231e-08\\
1.34375	5.34832648169115e-08\\
1.421875	5.34808249907925e-08\\
1.5	5.34836133991856e-08\\
1.578125	5.34889244563352e-08\\
1.65625	5.34948207675523e-08\\
1.734375	5.34998779444518e-08\\
1.8125	5.35031277892895e-08\\
1.890625	5.35038313931313e-08\\
1.96875	5.35015070857181e-08\\
2.046875	5.34958045916856e-08\\
2.125	5.34864946222235e-08\\
2.203125	5.34734575230456e-08\\
2.28125	5.34566350907095e-08\\
2.359375	5.34360418136259e-08\\
2.4375	5.34117292616543e-08\\
2.515625	5.33837965221995e-08\\
2.59375	5.33523805967828e-08\\
2.671875	5.33176447992112e-08\\
2.75	5.32797719277056e-08\\
2.828125	5.32389824170476e-08\\
2.90625	5.31954933991052e-08\\
2.984375	5.31495482347655e-08\\
3.0625	5.310143824655e-08\\
3.140625	5.30513937502253e-08\\
3.21875	5.29997097598045e-08\\
3.296875	5.29466835930137e-08\\
3.375	5.28926057119516e-08\\
3.453125	5.28377566977323e-08\\
3.53125	5.27824549345635e-08\\
3.609375	5.27269893579874e-08\\
3.6875	5.26716544824168e-08\\
3.765625	5.2616743823064e-08\\
3.84375	5.25625497016513e-08\\
3.921875	5.2509354697694e-08\\
4	5.24574266524969e-08\\
4.078125	5.2407031936319e-08\\
4.15625	5.23584313683045e-08\\
4.234375	5.23118506012832e-08\\
4.3125	5.22675327463418e-08\\
4.390625	5.22256698165524e-08\\
4.46875	5.21864886027235e-08\\
4.546875	5.21501560546422e-08\\
4.625	5.21168337375144e-08\\
4.703125	5.20866682007792e-08\\
4.78125	5.20597895348196e-08\\
4.859375	5.20363060418916e-08\\
4.9375	5.20162965200743e-08\\
};
\addlegendentry{Strang splitting}

\end{axis}

\end{tikzpicture}%
}
\end{tabular}
\caption{Comparison of methods of order 2 for mass and
    energy conservations: On the left hand side, variation of the mass with
    respect to time; on the right hand side,
    variation of the energy with respect to time}
  \label{fig:conservations_NLS1Dsoliton}
\end{figure}

\subsubsection{Two dimensional nonlinear Schr\"odinger equation}
\label{subsubsec:Schro2d}

In this section, we consider the following 2D nonlinear Schr\"odinger equation:
\begin{equation}
\label{eq:NLS2D}
i\partial_t u=-\Delta u - |u|^2 u, 
\end{equation}
with homogeneous Dirichlet boundary conditions on the domain represented in gray in the Figure
\ref{fig:2DdomainNLS} with $l_x=l_y=1,~p_x=2$ and $p_y=3$. The initial datum we chose reads
\begin{equation}
  \label{eq:iniNLS2D}
  u_0(x,y) = \sin\left(2\pi x\right)\sin\left(2\pi y\right)\exp\left(2i\pi x\right),
\end{equation}
when $(x,y)$ belongs to the domain.
In this particular case, the spectrum of the Laplace operator is not accessible
and one cannot use efficiently spectral methods such as exponential Runge--Kutta methods
or Lawson methods \cite{Nous2017}.

\begin{figure}[!h]
\centering
\begin{tikzpicture}
\draw[black](0,0)--(2,0)--(2,3)--(0,3)--cycle;
\draw[black,fill=gray!20](0,0)--(2,0)--(2,2)--(1,2)--(1,3)--(0,3)--cycle;
\draw[black](1,0)--(1,2);
\draw[black](0,2)--(1,2);
\draw[black](0,1)--(2,1);
\draw (1/2,5/2)node[]{$1$};
\draw (1/2,3/2)node[]{$3$};
\draw (1/2,1/2)node[]{$5$};
\draw (3/2,1/2)node[]{$6$};
\draw (3/2,3/2)node[]{$4$};
\draw (3/2,5/2)node[]{$2$};
\draw[<->] (0,-0.1)--(2,-0.1);
\draw (1,-0.1)node[below]{$L_x=p_xl_x$};
\draw[<->] (2.1,2)--(2.1,3);
\draw (2.1,5/2)node[right]{$l_y$};
\draw[<->] (1,3.1)--(2,3.1);
\draw (1.5,3.1)node[above]{$l_x$};
\draw[<->] (-0.1,0)--(-0.1,3);
\draw (-0.1,3/2)node[left]{$L_y=p_yl_y$};
\end{tikzpicture}
\caption{2D domain used in the simulation of the nonlinear Schr\"odinger equation \eqref{eq:NLS2D}}
\label{fig:2DdomainNLS}
\end{figure}
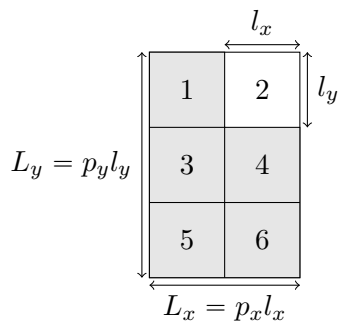

We use a finite differences discretization in space with, for $J \in \N^*$, $p_xJ+1$ points in the $x$-direction
and $p_yJ+1$ points in the $y$-direction in such a way that the step is the same in the two directions. 
This way, the numerical unknown $u_n$ at time $t_n$ is a vector of $\mathcal{N}=((p_y-1)J-1)\times (p_xJ-1)+J\times((p_x-1)J-1)$
complex numbers. No matter the way we label the unkowns,
the matrix $B$ of the Laplace operator with homogeneous Dirichlet boundary conditions is a
sparse matrix of size $\mathcal{N}\times \mathcal{N}$. For the numerical simulations we use $J=50$ which gives $\mathcal{N}=12251$ unknowns.
Moreover, we consider $T=0.5$ as a final time.
The methods we consider are
the two linearly implicit methods of order 2 introduced in Section \ref{subsec:examples} (one with Gauss points,
the other one with uniform points),
the Crank-Nicolson scheme and
the Strang splitting method.
As one has no direct access to the exact solution of \eqref{eq:NLS2D} with initial condition
\eqref{eq:iniNLS2D}, we precompute as a reference solution the numerical solution provided
by a Runge--Kutta method at Gauss points with $5$ stages (which therefore has order $10$)
with a time step of $10^{-3}$.
We initialize the linearly implicit methods with $\gamma_{-1+c_1}$ and $\gamma_{-1+c_2}$ computed using one step of a backward Crank-Nicolson scheme. 
Our numerical results are displayed in Figure \ref{fig:ordre2_NLS2D}.
As expected, the order of each method above is 2. For the two linearly implicit methods,
this again illustrates that the results of Theorem \ref{th:conv} extend numerically to this PDE case.
Note that, the constant of order is really better for the linearly implicit method with Gauss points
than all the other ones.
Moreover the CPU time required to achieve a given precision
is also smaller for the linearly implicit method with Gauss points. This is the first example where a linearly implicit method developped in this paper clearly outperforms implicit as well as explicit standard methods from the literature.


\begin{figure}[!h]
  \centering
   \begin{tabular}{cc}
\resizebox{0.48\textwidth}{!}{
%
%
\definecolor{mycolor1}{rgb}{0.00000,0.44700,0.74100}%
\definecolor{mycolor2}{rgb}{0.2,0.6,0.4}
\begin{tikzpicture}

\begin{axis}[%
width=4.521in,
height=3.566in,
at={(1.011in,0.642in)},
scale only axis,
xmin=-6.2,
xmax=-4.8,
xlabel style={font=\color{white!15!black}},
xlabel={Logarithm of the time step},
ymin=-9,
ymax=-2,
ylabel style={font=\color{white!15!black}},
ylabel={Logarithm of the final error},
axis background/.style={fill=white},
legend style={at={(0.03,0.97)}, anchor=north west, legend cell align=left, align=left, draw=white!15!black}
]
\addplot [color=blue, draw=none, mark size=6.0pt, mark=triangle, mark options={solid, rotate=180, blue}]
  table[row sep=crcr]{%
-4.80617997398389	-2.75751785759157\\
-5.10720996964787	-3.1722448450079\\
-5.40823996531185	-3.74927496376756\\
-5.70926996097583	-4.34966881849018\\
-6.01029995663981	-4.95162359069752\\
};
\addlegendentry{Crank--Nicolson}

\addplot [color=red, draw=none, mark size=6.0pt, mark=o, mark options={solid, red}]
  table[row sep=crcr]{%
-4.80617997398389	-5.88233194617328\\
-5.10720996964787	-6.64924933337887\\
-5.40823996531185	-7.29763640791663\\
-5.70926996097583	-7.91452233531015\\
-6.01029995663981	-8.52968314167835\\
};
\addlegendentry{Linearly implicit (Gauss)}

\addplot [color=mycolor2, draw=none, mark size=6.0pt, mark=diamond, mark options={solid, mycolor2}]
  table[row sep=crcr]{%
-4.80617997398389	-2.75739134557851\\
-5.10720996964787	-3.17212171815724\\
-5.40823996531185	-3.74914608465309\\
-5.70926996097583	-4.34953953389459\\
-6.01029995663981	-4.95149427939416\\
};
\addlegendentry{Strang Splitting}

\addplot [color=black, draw=none, mark size=6.0pt, mark=x, mark options={solid, black}]
  table[row sep=crcr]{%
-4.80617997398389	-2.75741350363057\\
-5.10720996964787	-3.17213233572018\\
-5.40823996531185	-3.74915417353356\\
-5.70926996097583	-4.34954726023174\\
-6.01029995663981	-4.95150195255786\\
};
\addlegendentry{Linearly implicit (uniform)}

\end{axis}
\end{tikzpicture}%
} 
& 
\resizebox{0.48\textwidth}{!}{
%
%
\definecolor{mycolor1}{rgb}{0.00000,0.44700,0.74100}%
\definecolor{mycolor2}{rgb}{0.2,0.6,0.4}
\begin{tikzpicture}

\begin{axis}[%
width=4.521in,
height=3.566in,
at={(1.011in,0.642in)},
scale only axis,
xmin=-9,
xmax=-2,
xlabel style={font=\color{white!15!black}},
xlabel={Logarithm of the final error},
ymin=0,
ymax=350000,
ylabel style={font=\color{white!15!black}},
ylabel={CPU time},
axis background/.style={fill=white},
legend style={legend cell align=left, align=left, draw=white!15!black}
]
\addplot [color=blue, mark size=6.0pt, mark=triangle, mark options={solid, rotate=180, blue}]
  table[row sep=crcr]{%
-2.75751785759157	12327.68\\
-3.1722448450079	25039.57\\
-3.74927496376756	51114.37\\
-4.34966881849018	100973.47\\
-4.95162359069752	208552.14\\
};
\addlegendentry{Crank--Nicolson}

\addplot [color=red, mark size=6.0pt, mark=o, mark options={solid, red}]
  table[row sep=crcr]{%
-5.88233194617328	17125.94\\
-6.64924933337887	34836\\
-7.29763640791663	72521.88\\
-7.91452233531015	147910.45\\
-8.52968314167835	313263.97\\
};
\addlegendentry{Linearly implicit (Gaus)}

\addplot [color=mycolor2, mark size=6.0pt, mark=diamond, mark options={solid, mycolor2}]
  table[row sep=crcr]{%
-2.75739134557851	4043.31\\
-3.17212171815724	8074.12\\
-3.74914608465309	16085.43\\
-4.34953953389459	33067.67\\
-4.95149427939416	68207.49\\
};
\addlegendentry{Strang Splitting}

\addplot [color=black, mark size=6.0pt, mark=x, mark options={solid, black}]
  table[row sep=crcr]{%
-2.75741350363057	5662.92\\
-3.17213233572018	11156.09\\
-3.74915417353356	22170.14\\
-4.34954726023174	45198.5\\
-4.95150195255786	91970.22\\
};
\addlegendentry{Linearly implicit (uniform)}

\end{axis}
\end{tikzpicture}%
}
  \end{tabular}
  \caption{Comparison of methods of order 2 applied to \eqref{eq:NLS2D}: On the left hand side, maximal numerical error as a function of the time step (logarithmic scales); on the right hand side, CPU time (in seconds) as a function of the maximal numerical error. } 
   \label{fig:ordre2_NLS2D}
\end{figure}

\begin{remark}[Preservation of mass in NLS equations by linearly implicit methods]
  \label{rem:masspres}
  If we consider a linearly implicit method defined by a Runge--Kutta collocation method of order $s$
  satisfying the Cooper condition
  \begin{equation}
    \label{eq:Cooper}
    \forall (i,j)\in\{1,\cdots,s\}^2,\qquad b_i b_j = b_i a_{i,j} + b_j a_{j,i},
  \end{equation}
  that is to say a Runge--Kutta collocation method at Gauss' points, then this linearly
  implicit method preserves the mass ({\it i.e.} the squared $L^2$-norm) for the NLS equation,
  regardless of the physical dimension of the problem.
  This property does not depend on the choice of the matrix $D$ and vector $(\theta_1,\cdots,\theta_s)$
  in step \eqref{eq:heritagegamma} as long as they are real-valued.
  This preservation property relies on the fact that,
  provided the initial $(\gamma_{-1+c_i})_{1\leq i\leq s}$
  are purely imaginary, so will be all the $(\gamma_{n+c_i})_{1\leq i\leq s}$.
  This property will be detailed further in a forthcoming paper dealing with the time integration
  of PDEs using linearly implicit methods.
  For example, the linearly implicit methods of order $1$, $4$ and $6$ from Section
  \ref{subsec:examples} do not satisfy the Cooper condition \eqref{eq:Cooper},
  hence they do not preserve the mass for the NLS equation.
  In contrast, the first method of order $2$ of Section \ref{subsec:examples},
  which uses Gauss' points, preserves the mass of the solution of the NLS equation
  (and the second method of order $2$ does not).
\end{remark}

\subsection{Application to the nonlinear heat equation}
\label{subsec:heat}

In the previous sections, we have proved and illustrated that the linearly implicit methods developed
in this paper have good quantitative properties.
The goal of this section is to illustrate that they can indeed also have good qualitative properties.
Indeed, on some nonlinear heat equation with gradient-flow structure, we give an example
below of a linearly implicit fully discrete scheme which preserves the nonnegativity of the
solution (just as the exact flow does) as well as the decay of a discrete energy
which is consistent with the continuous energy of the problem.

Let us consider the one dimensional nonlinear heat equation given by
\begin{equation}
\label{eq:NLH}
\partial_t u = \partial_x^2 u +  u^3,
\end{equation}
with homogeneous Dirichlet boundary conditions on $\Omega=(-50,50)$.
This corresponds to equation \eqref{eq:PDE} with $L=\partial_x^2$ and $N(u)=u^2$.
Equation \eqref{eq:NLH} is the $L^2$-gradient flow equation for the energy:
\begin{equation}
\label{eq:NRJNLH}
E(u)=\dfrac{1}{2} \int_\Omega (\partial_x u)^2 \dd  x -\dfrac{1}{4} \int_\Omega u^4 \dd x,
\end{equation}
defined for $u \in H^1_0(\Omega)$.
It is well-known in the literature (see \cite{Hayakawa73} for example) that

\begin{prop}
\label{prop:NLH}
For all $u_0 \in H^1_0(\Omega),~u_0 \neq 0$,
the equation \eqref{eq:NLH} has a unique maximal solution $u$ in
$\mathcal{C}^0([0,T_*),H^1_0(\Omega)) \cap \mathcal{C}^1((0,T_*),L^2(\Omega))$ with $u(0)=u_0$
for some $T_* >0$.
Moreover this solution $u$ satisfies 
\begin{equation}
\label{eq:DecNRJ_NLH}
\forall ~t \in (0,T_*), \qquad  \dfrac{\dd E(u(t))}{\dd t} \leq 0.
\end{equation}
Finally if $u_0 \geq 0$ on $\Omega$ then for all $t \in [0,T_*)$, $u(t) \geq 0$ on $\Omega$.
\end{prop}

In order to give an example with good qualitative properties,
we consider a modified version of the fully discrete one stage method presented in Section \ref{subsec:examples} with $s=1$, $c_1=1/2$ so that $a_{1,1}=1/2$ and $b_1=1/2$, and with $\lambda_1=1/2$ so that $y_1=1/2$ and $\theta_1=1/2$.

Let us denote by $\mathcal{N}$ the number of unknowns, so that $\delta x = 100/(\mathcal{N}+1)$. We denote by
$\langle\cdot,\cdot\rangle$ the scalar product on $\R^\mathcal{N}$ defined for $v,w\in\R^\mathcal{N}$ by
$\langle v,w\rangle=\delta x\sum_{k=1}^\mathcal{N} v(k)w(k)$ and by $\| \cdot \|_2$ the associated norm. Moreover, for all $v\in\R^\mathcal{N}$, we denote by $v^{\circ 2}$ the vector of $\R^\mathcal{N}$ with component $k$
equal to $v^{\circ 2}(k)=v(k)^2$.

Then, the stage \eqref{eq:heritagegamma} reads here
\begin{equation}
\label{eq:heritagegammaNLH}
\gamma_{n+1/2} = \dfrac{1}{2} \gamma_{n-1/2} + \dfrac{1}{2} u_n^{\circ 2},
\end{equation}
and the stages \eqref{eq:RKimpl}, and \eqref{eq:RKexpl} can be summarized by
\begin{equation}
\label{eq:schemeNLH}
\dfrac{u_{n+1}-u_n}{h}=\left(B +{\rm diag}(\gamma_{n+1/2})\right) \dfrac{u_{n+1}+u_n}{2},
\end{equation}
where $B$ denotes the matrix of the Laplacian operator with homogeneous Dirichlet boundary conditions
on $\Omega$ on the equispaced grid, with space step size $\delta x$,
as defined after \eqref{eq:approxexp}. We still denote by $u_0$ the evaluation of the initial datum $u_0$ on the equispaced grid.
In addition, we choose for $\gamma_{-1/2}$ the evaluation of $N(u_0)$ on the same grid. 

The fully discrete energy associated to the numerical scheme is defined for $u,\gamma\in\R^\mathcal{N}$ by
\begin{equation}
\label{eq:discNRJNLH}
E_{rlx}(u,\gamma)=-\dfrac12 \langle u,B u\rangle - \dfrac{1}{2} \langle\gamma,u^{\circ 2}\rangle +
\dfrac{1}{4} \langle\gamma,\gamma\rangle.
\end{equation}
Note that this formula is consistent with the continuous energy $E$ defined in \eqref{eq:NRJNLH}. 

\begin{prop}
\label{prop:discNLH}
Let us assume $u_0 \in H^1_0(\Omega)$. Still denote by $u_0$ the projection of $u_0$ onto
the equispaced grid with $\mathcal{N}$ interior points. Choose $T \in (0,T_*)$.

1- Let us assume that there exists $h_0,\delta x_0>0$ such that for all $h \in (0,h_0)$
and all $\delta x\in(0,\delta x_0)$ with $h<\delta x^2$,
the sequence $(\gamma_{n+1/2})_{n\geq 0}$ is bounded in $\R^\mathcal{N}$ with the maximum norm
as long as $(n+1/2)h\leq T$.
Then, for all $h \in (0,h_0)$, $\delta x \in(0,\delta x_0)$ and $n$ such that $(n+1/2)h\leq T$
and $h/\delta x^2<1$, $\gamma_{n+1/2}$ is a nonnegative real-valued vector. 
Moreover assuming $u_0 \geq 0$, there exists a constant $h_1 \in (0,h_0)$ such that for all
$h \leq h_1$ and $\delta x\in(0,\delta x_0)$ with $h/\delta x^2<1$,
the sequence $(u_n)_{n \geq 0}$ is a sequence of nonnegative vectors as long as $nh\leq T$. 

2- For all $h \in (0,h_0)$ and for all $n\in\N$ such that $(n+1)h\leq T$,
the sequence $(u_n,\gamma_{n-1/2})_{n \geq 0}$ satisfies
\begin{equation}
\label{eq:DecNRJdiscrNLH}
E_{rlx}(u_{n+1},\gamma_{n+1/2}) \leq  E_{rlx}(u_n,\gamma_{n-1/2}).
\end{equation}
\end{prop}

\begin{proof}
  We will use M, P and Z matrices as defined for example in the Chapter 10 of \cite{BermanPlemmons87}.
Let us give the main ideas of the proof of proposition \ref{prop:discNLH}. 

1- The sign of $\gamma_{n+1/2}$ is a direct consequence of \eqref{eq:heritagegammaNLH} and the choice of the initial condition $\gamma_{-1/2}=N(u_0)$
: it is a convex combination of vectors with same signs. 
Moreover assuming $u_0 \geq 0$, the nonnegativity of $u_n$ can be obtained by induction using the following arguments. The equation \eqref{eq:schemeNLH} can be written 
\begin{equation}
  \label{eq:NLHinterm}
\left(1-\dfrac{h}{2}B -\dfrac{h}{2} {\rm diag}(\gamma_{n+1/2})\right)u_{n+1}=\left(1+\dfrac{h}{2}B +\dfrac{h}{2} {\rm diag}(\gamma_{n+1/2})\right)u_n.
\end{equation}
Since $h/\delta x^2<1$ and $u_n\geq 0$, one has $\left(1+\dfrac{h}{2}B\right)u_n\geq 0$.
Since $u_n\geq 0$ and $\gamma_{n+1/2}\geq 0$, one has that ${\rm diag}(\gamma_{n+1/2})u_n\geq 0$,
so that the right-hand side of \eqref{eq:NLHinterm} is nonnegative componentwise.
Moreover, the operator in the left-hand side of \eqref{eq:NLHinterm} has nonnegative inverse
since it is an M-matrix for $h_1\in(0,h_0)$ small enough (depending on the bound on the
maximum norm of the sequence $(\gamma_{n+1/2})_{n\geq 0}$). Indeed, one can check
that it is a Z-matrix since its off-diagonal coefficients are nonpositive, and it is also
a P-matrix (for $h\in(0,h_1)$).

2- Taking the scalar product of \eqref{eq:schemeNLH} with $u_{n+1}-u_n$
we obtain
\begin{equation*}
  \dfrac{1}{h}  \| u_{n+1}-u_n \|_2^2
  = \dfrac{1}{2} \langle u_{n+1},Bu_{{n+1}}\rangle
  - \dfrac{1}{2} \langle u_n,Bu_n\rangle
  + \dfrac{1}{2} \langle \gamma_{n+1/2},u_{n+1}^{\circ 2}-u_n^{\circ 2}\rangle,
\end{equation*}
which gives
\begin{eqnarray*}
    \dfrac{1}{h}  \| u_{n+1}-u_n \|_2^2
&=& -E_{rlx}(u_{n+1},\gamma_{n+1/2})+ E_{rlx}(u_n,\gamma_{n-1/2}) \\
&+ & \langle \gamma_{n+1/2},\dfrac{1}{4}\gamma_{n+1/2}-\dfrac{1}{2}u_n^{\circ 2}\rangle + \langle \gamma_{n-1/2}, -\dfrac{1}{4}\gamma_{n-1/2}+\dfrac{1}{2}u_n^{\circ 2}\rangle. 
\end{eqnarray*}
Then, using \eqref{eq:heritagegammaNLH}, a straightforward computation leads to 
\begin{equation}
\label{eq:idNRJdiscrNLH}
\dfrac{1}{h}  \| u_{n+1}-u_n \|_2^2 + \dfrac{3}{4} \| \gamma_{n+1/2}-\gamma_{n-1/2}\|_2^2= -E_{rlx}(u_{n+1},\gamma_{n+1/2})+ E_{rlx}(u_n,\gamma_{n-1/2}),
\end{equation}
which implies the result \eqref{eq:DecNRJdiscrNLH}.
\end{proof}

We display in Figure \ref{fig:ordre1_chaleur} the comparison of the method above
with the Lie splitting method,
with the linear part approximated by a formula similar to \eqref{eq:approxexp}, and with the implicit
Euler method. We compute the $L^2$ numerical errors using a reference solution obtained
by a standard method of order 10 with a very small time step.
Unsurprisingly, the three methods are of order 1 numerically. Moreover, the linearly implicit method
is faster than the implicit Euler method for a given error, but slower than the Lie splitting method.
Note that all the methods preserve the nonnegativity of the solution (as long as one has
a bound on $\|u_n\|_2$ and $h$ is sufficiently small with respect to this bound, and under
an additional CFL condition for the Lie splitting method).

\begin{figure}[!h]
  \centering
   \begin{tabular}{cc}
\resizebox{0.48\textwidth}{!}{
%
%
\definecolor{mycolor2}{rgb}{0.2,0.6,0.4}

\begin{tikzpicture}

\begin{axis}[%
width=4.521in,
height=3.566in,
at={(1.011in,0.642in)},
scale only axis,
xmin=-5.4,
xmax=-3.4,
xlabel style={font=\color{white!15!black}},
xlabel={Logarithm of the time step},
ymin=-8,
ymax=-3,
ylabel style={font=\color{white!15!black}},
ylabel={Logarithm of the final error},
axis background/.style={fill=white},
legend style={at={(0.03,0.97)}, anchor=north west, legend cell align=left, align=left, draw=white!15!black}
]
\addplot [color=red, draw=none, mark size=6.0pt, mark=o, mark options={solid, red}]
  table[row sep=crcr]{%
-3.40823996531185	-3.20611759524951\\
-3.70926996097583	-3.50673882113964\\
-4.01029995663981	-3.80756424217899\\
-4.31132995230379	-4.10849190292709\\
-4.61235994796777	-4.40947071938905\\
-4.91338994363176	-4.71047512204763\\
-5.21441993929574	-5.01149232067857\\
};
\addlegendentry{Linearly implicit}

\addplot [color=blue, draw=none, mark size=6.0pt, mark=triangle, mark options={solid, rotate=180, blue}]
  table[row sep=crcr]{%
-3.40823996531185	-3.50741693386784\\
-3.70926996097583	-3.80855477333897\\
-4.01029995663981	-4.10963871138263\\
-4.31132995230379	-4.41069547985917\\
-4.61235994796777	-4.71173878126624\\
-4.91338994363176	-5.01278004748716\\
-5.21441993929574	-5.31380935531079\\
};
\addlegendentry{Implicit Euler}

\addplot [color=mycolor2, draw=none, mark size=6.0pt, mark=diamond, mark options={solid, mycolor2}]
  table[row sep=crcr]{%
-3.40823996531185	-6.14920662287822\\
-3.70926996097583	-6.4502220163738\\
-4.01029995663981	-6.75126157480311\\
-4.31132995230379	-7.0523265723831\\
-4.61235994796777	-7.35315247898513\\
-4.91338994363176	-7.65446091046005\\
-5.21441993929574	-7.95421701509056\\
};
\addlegendentry{Lie Splitting}

\end{axis}
\end{tikzpicture}%
} 
& 
\resizebox{0.48\textwidth}{!}{
%
%
\definecolor{mycolor2}{rgb}{0.2,0.6,0.4}

\begin{tikzpicture}

\begin{axis}[%
width=4.521in,
height=3.566in,
at={(1.011in,0.642in)},
scale only axis,
xmin=-8,
xmax=-3,
xlabel style={font=\color{white!15!black}},
xlabel={Logarithm of the final error},
ymin=-1,
ymax=2,
ylabel style={font=\color{white!15!black}},
ylabel={CPU time},
axis background/.style={fill=white},
legend style={at={(0.03,0.97)}, anchor=north west, legend cell align=left, align=left, draw=white!15!black}
]
\addplot [color=red, mark size=6.0pt, mark=o, mark options={solid, red}]
  table[row sep=crcr]{%
-3.20611759524951	-0.16115090926271\\
-3.50673882113964	-0.267606240177061\\
-3.80756424217899	0.041392685158234\\
-4.10849190292709	0.338456493604617\\
-4.40947071938905	0.63848925695464\\
-4.71047512204763	0.941014243705571\\
-5.01149232067857	1.24229290498293\\
};
\addlegendentry{Linearly implicit}

\addplot [color=blue, mark size=6.0pt, mark=triangle, mark options={solid, rotate=180, blue}]
  table[row sep=crcr]{%
-3.50741693386784	0.292256071356459\\
-3.80855477333897	0.579783596616818\\
-4.10963871138263	0.861534410859037\\
-4.41069547985917	1.16405529189345\\
-4.71173878126624	1.464191370641\\
-5.01278004748716	1.71667097556014\\
-5.31380935531079	1.88840416773705\\
};
\addlegendentry{Implicit Euler}

\addplot [color=mycolor2, mark size=6.0pt, mark=diamond, mark options={solid, mycolor2}]
  table[row sep=crcr]{%
-6.14920662287822	-0.638272163982373\\
-6.4502220163738	-0.397940008672062\\
-6.75126157480311	-0.102372908709579\\
-7.0523265723831	0.2041199826559\\
-7.35315247898513	0.501059262217746\\
-7.65446091046005	0.806179973983886\\
-7.95421701509056	1.10653085382238\\
};
\addlegendentry{Lie Splitting}

\end{axis}
\end{tikzpicture}%
}
  \end{tabular}
  \caption{Comparison of methods of order 1 applied to \eqref{eq:NLH}: On the left hand side, maximal numerical error as a function of the time step (logarithmic scales); on the right hand side, CPU time (in seconds) as a function of the maximal numerical error. } 
   \label{fig:ordre1_chaleur}
\end{figure}

We display in Figure \ref{fig:NRJdecay_chaleur} the plots of the initial datum and the final time
solution obtained at $T=1$ with the same linearly implicit method of order 1
(for $h=1/(5\times2^{11})$) (left hand side) and the plot of the evolution of $E_{rlx}$ (right hand side).
This illustrates the results of Proposition \ref{prop:discNLH} : the numerical solution
starting from a nonnegative initial datum stays nonnegative, and the discrete energy does
not increase with time (see \eqref{eq:DecNRJdiscrNLH}).

\begin{figure}[!h]
  \centering
   \begin{tabular}{cc}
\resizebox{0.48\textwidth}{!}{
%
%
\begin{tikzpicture}

\begin{axis}[%
width=4.521in,
height=3.566in,
at={(1.011in,0.642in)},
scale only axis,
xmin=-50,
xmax=50,
ymin=0,
ymax=0.8,
axis background/.style={fill=white},
legend style={legend cell align=left, align=left, draw=white!15!black}
]
\addplot [color=red, mark=o, mark size=3.0pt, mark options={solid, red}]
  table[row sep=crcr]{%
-49.9024390243902	0.00153248182189678\\
-47.9512195121951	0.0321599527061849\\
-46	0.0626666167821522\\
-44.0487804878049	0.0929378777146235\\
-42.0975609756098	0.122860023445298\\
-40.1463414634146	0.152320653344831\\
-38.1951219512195	0.181209100438624\\
-36.2439024390244	0.20941684712024\\
-34.2926829268293	0.236837932790868\\
-32.3414634146341	0.263369351893539\\
-30.390243902439	0.288911440846927\\
-28.4390243902439	0.313368252425222\\
-26.4878048780488	0.33664791617776\\
-24.5365853658537	0.3586629835345\\
-22.5853658536585	0.379330756300994\\
-20.6341463414634	0.398573597308872\\
-18.6829268292683	0.416319222054907\\
-16.7317073170732	0.432500970233133\\
-14.7804878048781	0.44705805614003\\
-12.8292682926829	0.459935797012132\\
-10.8780487804878	0.47108581843834\\
-8.92682926829269	0.480466236075306\\
-6.97560975609756	0.488041812983293\\
-5.02439024390244	0.493784091991491\\
-3.07317073170732	0.497671502595557\\
-1.1219512195122	0.499689441985842\\
0.829268292682928	0.499830329901911\\
2.78048780487805	0.498093637107304\\
4.73170731707317	0.494485887377584\\
6.68292682926829	0.489020632994188\\
8.63414634146341	0.481718403836147\\
10.5853658536585	0.472606630260909\\
12.5365853658537	0.46171954006395\\
14.4878048780488	0.449098029904248\\
16.4390243902439	0.434789511678598\\
18.390243902439	0.418847734421844\\
20.3414634146341	0.401332582402068\\
22.2926829268293	0.38230985016916\\
24.2439024390244	0.361850995401791\\
26.1951219512195	0.340032870481212\\
28.1463414634146	0.316937433800188\\
30.0975609756098	0.29265144189153\\
32.0487804878049	0.267266123532726\\
34	0.240876837050858\\
35.9512195121951	0.213582712115141\\
37.9024390243902	0.185486277362651\\
39.8536585365854	0.156693075256048\\
41.8048780487805	0.127311265620057\\
43.7560975609756	0.0974512193459911\\
45.7073170731707	0.067225103790552\\
47.6585365853659	0.0367464614263153\\
49.609756097561	0.00612978332668585\\
};
\addlegendentry{Initial datum}

\addplot [color=blue, mark=diamond, mark size=3.0pt, mark options={solid, blue}]
  table[row sep=crcr]{%
-49.9024390243902	0.00153210499632899\\
-47.9512195121951	0.0321851921736069\\
-46	0.0628979385886272\\
-44.0487804878049	0.0937258473261866\\
-42.0975609756098	0.124721767789927\\
-40.1463414634146	0.155934333441097\\
-38.1951219512195	0.187406073383418\\
-36.2439024390244	0.219171074167908\\
-34.2926829268293	0.25125205762889\\
-32.3414634146341	0.283656732972568\\
-30.390243902439	0.316373282349563\\
-28.4390243902439	0.349364857492989\\
-26.4878048780488	0.382563013500171\\
-24.5365853658537	0.415860101940413\\
-22.5853658536585	0.449100810719143\\
-20.6341463414634	0.482073294978443\\
-18.6829268292683	0.514500707408546\\
-16.7317073170732	0.54603440203672\\
-14.7804878048781	0.576250603945866\\
-12.8292682926829	0.60465279061859\\
-10.8780487804878	0.630682215153438\\
-8.92682926829269	0.653738644335408\\
-6.97560975609756	0.673212221410433\\
-5.02439024390244	0.688525293621636\\
-3.07317073170732	0.699180320248777\\
-1.1219512195122	0.704807311116746\\
0.829268292682928	0.705202683691343\\
2.78048780487805	0.700351916409026\\
4.73170731707317	0.690431198651352\\
6.68292682926829	0.675787718201368\\
8.63414634146341	0.656902792484925\\
10.5853658536585	0.634345185451254\\
12.5365853658537	0.608722796331441\\
14.4878048780488	0.580639601982165\\
16.4390243902439	0.550662163889866\\
18.390243902439	0.519297238842822\\
20.3414634146341	0.486979837387712\\
22.2926829268293	0.454069774953913\\
24.2439024390244	0.420854297311693\\
26.1951219512195	0.387554483759931\\
28.1463414634146	0.354333559474671\\
30.0975609756098	0.321305766574183\\
32.0487804878049	0.288544921566731\\
34	0.256092167316507\\
35.9512195121951	0.223962699921648\\
37.9024390243902	0.192151428505724\\
39.8536585365854	0.16063763095172\\
41.8048780487805	0.129388723017383\\
43.7560975609756	0.0983632802161677\\
45.7073170731707	0.0675134548127188\\
47.6585365853659	0.0367869235582641\\
49.609756097561	0.00612849141141468\\
};
\addlegendentry{Final time solution}

\end{axis}
\end{tikzpicture}%
} 
& 
\resizebox{0.48\textwidth}{!}{
\input{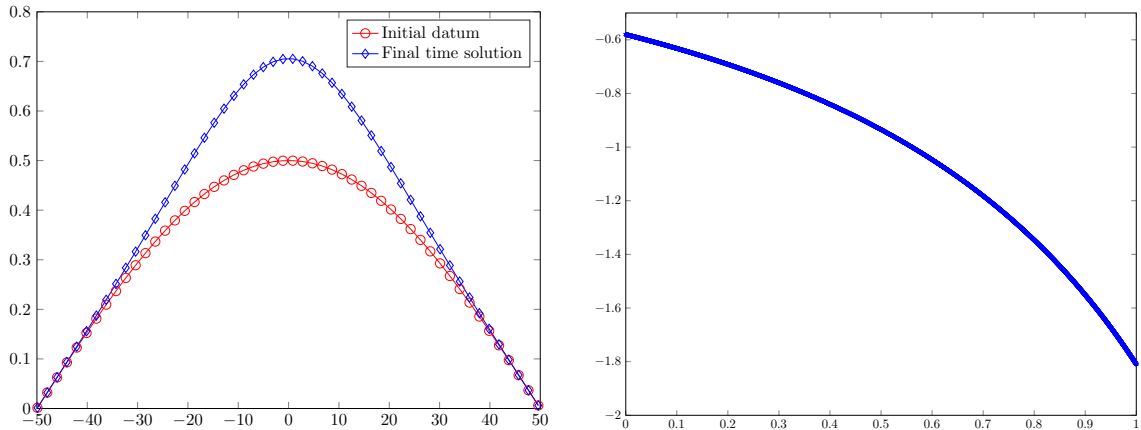}
}
  \end{tabular}
  \caption{Initial datum and solution at the final time $T=1$ with respect to space (left hand side) and evolution of $E_{rlx}$ with respect to time (right hand side).} 
   \label{fig:NRJdecay_chaleur}
\end{figure}

\section{Conclusion and perspectives}
This paper introduces a new class of methods for the time integration of evolution problems
set as systems of ODEs (or PDEs after space discretization).
This class contains methods that are only linearly implicit, no matter the evolution equation.
Moreover, the paper describes a specific way to design linearly implicit methods of any arbitrarily
high order.
Using suitable definitions of consistency and stability, we prove that such methods
are actually of high order for ODEs, and the proof extends to finite systems of ODEs.
We illustrate numerically that some of these methods
are of the expected order for two examples of PDEs (nonlinear Schr\"odinger equation in 1d and 2d
and a nonlinear heat equation in 1d), and discuss some of their qualitative properties.
Our numerical results show that the linearly implicit methods introduced in this paper
behave rather poorly in terms of efficiency for simple small systems of ODEs.
In contrast, they illustrate numerically that a linearly implicit method of order $2$ outperforms
standard methods of order $2$ from the litterature for a NLS equation on a domain where no spectral
method can be applied.

Perspectives of this work include a rigorous analysis of these linearly implicit methods
in PDE contexts ({\it i.e.} before discretization in space).
In this direction, a recent result \cite{BDDLV19} proves that the relaxation method
\eqref{eq:keepcoolBesse}, which belongs to the class of methods presented here
(see remark \ref{rem:relaxBesse}), is of order $2$ when applied to the NLS equation.
Another question is that of the
possibility to design, in a systematic way,
linearly implicit methods of high order with some qualitative properties
adapted to the PDE problem ({\it e.g.} preservation of mass or energy for NLS equation,
energy decrease for parabolic problems).

\section*{Acknowledgements}
  This work was partially supported by the Labex CEMPI (ANR-11-LABX-0007-01).

\bibliographystyle{spmpsci}

\bibliography{labib}

\end{document}